\documentclass[twoside,10pt]{article}
\usepackage[colorlinks]{hyperref}                        
\usepackage{amsmath,latexsym,amsfonts,indentfirst,amsthm,amsxtra,amssymb,bm}

\usepackage[perpage,symbol*]{footmisc}

\numberwithin{equation}{section}

 \newtheorem{lemma}{Lemma}[section]
 \newtheorem{proposition}[lemma]{Proposition}
 \newtheorem{theorem}{Theorem}
 \newtheorem{corollary}[lemma]{Corollary}

 \theoremstyle{remark}
 \newtheorem{remark}{Remark}[section]

\begin{document}

\title{\bf Stability of  stationary solutions to the
outflow problem for full compressible Navier-Stokes equations
with large initial perturbation}
\author{{\bf Ling Wan}\,$^{\rm a}$,
{\bf Tao Wang}\,$^{\rm a,}$\thanks{Corresponding author.
E-mail address: tao.wang@whu.edu.cn}\,\,,
{\bf Qingyang Zou}\,$^{\rm b}$}


\maketitle
\begin{center}
$^{\rm a}$\,\emph{\small
School of Mathematics and Statistics, Wuhan University,
Wuhan 430072, China}
\\
$^{\rm b}\,$\emph{\small
School of Science, Wuhan University of Science and Technology,
Wuhan 430081, China}

\end{center}

\begin{abstract}
We investigate the large-time behavior of solutions
to an outflow problem of
the full compressible Navier-Stokes equations in the half line.
The non-degenerate stationary solution is shown to be
asymptotically stable
under large initial perturbation with no restriction
on the adiabatic exponent
$\gamma$, provided that the boundary strength is sufficiently small.
The proofs are based on the standard
energy method and the crucial step is to
obtain positive lower and upper bounds of the density and the temperature
uniformly in time and space.

  \bigbreak
  {\bf Keywords:} full compressible Navier-Stokes equations;
outflow problem; stationary solution; stability; large initial perturbation
  \bigbreak
  {\bf Mathematics Subject Classification:} 76N10 (35B35 35B40 35Q35)
\end{abstract}

\section{Introduction}
We study an initial-boundary value problem for the full compressible
Navier-Stokes equations
\begin{equation} \label{NS_E}
  \left\{
  \begin{aligned}
    \rho_t+(\rho u)_x&=0,\\[0.5mm]
    (\rho u)_t+(\rho u^2+ P)_x&=(\mu u_x)_x,\\[0.5mm]
    (\rho{E})_t+(\rho u{E}+uP)_x&
    =(\kappa\theta_x+\mu uu_x)_x,
  \end{aligned}
  \right.
\end{equation}
where $t>0$ is the time variable, $x\in\mathbb{R}_+:=(0,\infty)$
is the spatial variable,
and the primary dependent variables are the density $\rho$, fluid velocity
$u$ and absolute temperature $\theta$. The specific total energy
${E}=e+\frac12u^2$
with $e$ being the specific internal energy.
The viscosity coefficient $\mu$ and
the heat conductivity $\kappa$ are assumed to be positive constants.
Here we focus on the ideal polytropic gas, that is, the pressure $P$ and
the specific internal energy $e$ are given by the
constitutive relations
\begin{equation}\label{id-po}
  P=R\rho\theta
,   \quad e=c_v\theta,
\end{equation}
where $R>0$ is the gas constant and the specific heat $c_v=R/(\gamma-1)$
with $\gamma>1$ being the adiabatic exponent.

The system \eqref{NS_E} is supplemented with the initial data
\begin{equation}\label{initial}
(\rho, u, \theta)|_{t=0}=(\rho_0, u_0, \theta_0),
\end{equation}
which are assumed to satisfy
\begin{equation}\label{initial1}
  \lim_{x\to\infty}(\rho_0, u_0, \theta_0)(x)=(\rho_+,u_+,\theta_+),
\end{equation}
where $\rho_+>0$, $u_{+}$ and $\theta_{+}>0$ are constants.
For boundary conditions, we take
\begin{equation}
  \label{bdy}
   u(t,0)=u_-,\quad \theta(t,0)=\theta_{-}>0,
\end{equation}
where $u_{-}<0$ and $\theta_-$ are constants.

The assumption $u_-<0$ means that the fluid blows out from the boundary,
and hence the problem \eqref{NS_E}-\eqref{bdy} is called the
outflow problem.
In the case of $u_-=0$ the problem is called
the impermeable wall problem,
while the problem \eqref{NS_E}-\eqref{bdy} with
the additional boundary condition $\rho(t,0)=\rho_-$
in the case of $u_->0$
is called the inflow problem.

Matsumura \cite{Ma01MR1944189} considered initial-boundary value
problems for the isentropic
Navier-Stokes equations in the half line $\mathbb{R}_+$
and proposed
a complete classification about
the precise description of the large-time behaviors of solutions.
Since then, some results have been obtained for
the rigorous mathematical justification of this classification.
For the impermeable wall problem,
the stability for the viscous shock wave
is obtained in \cite{MM99MR1682659} for small initial perturbation,
while Matsumura-Nishihara \cite{MN00MR1738558} showed that
the rarefaction wave is asymptotically stable under large initial perturbation.
For the inflow problem under small initial perturbation,
see \cite{MN01MR1888084} for
the stability of the boundary layer solution and
its superposition with the rarefaction wave,
and see \cite{HMS03MR1997442} for
the stability of the viscous shock
wave and its superposition  with the boundary layer solution.
Recently, Fan et al. \cite{FLWZMR3260233} established the asymptotic stability of both
the boundary layer solution and the supersonic rarefaction wave for
a certain class of large initial perturbation.
For the outflow problem,
the stability of the stationary solution
and its superposition with the rarefaction wave
is proved in
\cite{KNZ03MR2005853,KZ08MR2420517}
under small initial perturbation,
while a convergence rate of solutions
toward the stationary solution was obtained in
\cite{NNYMR2356211} provided that the initial
perturbation belongs to some weighted Sobolev space.
Huang-Qin \cite{HQ09MR2514736} improved the results in
\cite{KNZ03MR2005853,KZ08MR2420517} to large initial perturbation.

There have been some works on
the large-time behaviour of solutions
to the initial-boundary value problems of the full compressible Navier-Stokes equations
\eqref{NS_E}-\eqref{id-po} in the half line.
For the inflow problem,
Qin-Wang \cite{QW09MR2578799,QW11MR2765694} showed the existence of the boundary layer solution
as well as the stability of the boundary layer
and its superposition with the viscous contact wave and rarefaction waves.
The asymptotic stability of the rarefaction wave, boundary layer solution,
and their superposition is proved in \cite{HLSMR2730324} for
both the impermeable wall problem and the inflow problem.
Kawashima et al. \cite{KNNZMR2755498} proved the existence, the nonlinear stability
and the convergence rate of the stationary solution for the outflow problem.
We note that
the results in \cite{HLSMR2730324,KNNZMR2755498,QW09MR2578799,QW11MR2765694} are all concerned
with small initial perturbation.
Thus a problem of interest is whether
stability results on the stationary solution and the rarefaction wave
hold for the outflow problem \eqref{NS_E}-\eqref{bdy} with
large initial perturbation as in \cite{HQ09MR2514736} for the isentropic model.
In this direction, motivated by \cite{NYZMR2083790},
Qin \cite{Q11MR2785974}
proved that the non-degenerate stationary solution is
asymptotically stable under ``partially''
large initial perturbation with the technical condition
that the adiabatic exponent $\gamma$ is close to 1.

In this article, we establish
the large-time behavior of solutions toward
stationary solutions
to the outflow problem \eqref{NS_E}-\eqref{bdy}
under
large initial perturbation
without any restriction on the adiabatic exponent $\gamma$.
We expect that the solution
converges to a stationary solution
$(\tilde{\rho},\tilde{u},\tilde{\theta})(x)$ of \eqref{NS_E}:
\begin{equation}\label{BL}
 \left\{
  \begin{aligned}
    (\tilde{\rho}\tilde{u})'&=0,\qquad\qquad\qquad ':=\frac{d}{dx},
    \quad x\in\mathbb{R}_+,\\
    (\tilde{\rho}\tilde{u}^2+ \tilde{P})'&=\mu \tilde{u}'',\\[0.9mm]
    (\tilde{\rho} \tilde{u}\tilde{E}
    +\tilde{u}\tilde{P})'
    &=\kappa\tilde{\theta}''+\mu (\tilde{u}\tilde{u}')',
  \end{aligned}
  \right.
\end{equation}
where
$\tilde{P}:=R\tilde{\rho}\tilde{\theta}$
and
$\tilde{E}:=c_v\tilde{\theta}+\frac12\tilde{u}^2$.
We also assume that the stationary solution satisfies
the far-field condition \eqref{initial1} and the
boundary condition \eqref{bdy}:
\begin{equation}\label{BL2}
      (\tilde{u},\tilde{\theta})(0)=(u_-,\theta_-),\quad
    \lim_{x\to\infty}(\tilde{\rho},\tilde{u},\tilde{\theta})(x)
    =(\rho_+,u_+,\theta_+).
\end{equation}
We define the boundary strength $\delta$ as
$$
\delta:=\left|(u_+-u_-,\theta_+-\theta_-)\right|,
$$
and the Mach number at infinity as
$$
M_+:=\frac{|u_+|}{c_+},
$$
where $c_+:=\sqrt{R\gamma\theta_+}$ is the sound speed.
The existence and the properties of the stationary solution
$(\tilde{\rho},\tilde{u},\tilde{\theta})$
satisfying \eqref{BL} and \eqref{BL2} are quoted in
the following lemma.
\begin{lemma}
  [Existence of stationary solution \cite{KNNZMR2755498}]
  Suppose that $(u_-,\theta_-)$ satisfies
  \begin{equation}\label{c1}
   (u_-,\theta_-)\in \mathcal{M}^+:=\left\{(u,\theta)\in\mathbb{R}^2:
    \left|(u-u_+,\theta-\theta_+)\right|<\delta_0\right\}
  \end{equation}
  for a certain positive constant $\delta_0$.
\begin{list}{}{\setlength{\parsep}{\parskip}
             \setlength{\itemsep}{\parskip}
             \setlength{\labelwidth}{2em}
             \setlength{\labelsep}{0.4em}
             \setlength{\leftmargin}{2.2em}
             \setlength{\topsep}{1mm}
             }
    \item[{\rm(i)}] For the case $M_+>1$, there exists a unique smooth solution
  $(\tilde{\rho},\tilde{u},\tilde{\theta})$ to the problem \eqref{BL}-\eqref{BL2}
  satisfying
  \begin{equation}\label{decay1}
   |\partial_x^n(\tilde{\rho}-\rho_+,\tilde{u}-u_+,\tilde{\theta}-\theta_+)(x)|
   \leq C\delta e^{-cx}
  \end{equation}
  for integer $n=0,1,2,\cdots,$
  where $C$ and $c$ are positive constants.
  \item[{\rm(ii)}] For the case $M_+=1$,
   there exists a certain region $\mathcal{M}^0\subset \mathcal{M}^+$
   such that if $(u_-,\theta_-)\in \mathcal{M}^0$, then there exists
   a unique smooth solution $(\tilde{\rho},\tilde{u},\tilde{\theta})$
to \eqref{BL}-\eqref{BL2}
satisfying
  \begin{equation}\label{decay2}
   |\partial_x^n(\tilde{\rho}-\rho_+,\tilde{u}-u_+,\tilde{\theta}-\theta_+)(x)|
   \leq \frac{C\delta^{n+1}}{(1+\delta x)^{n+1}}
   +C\delta e^{-cx}
  \end{equation}
  for integer $n=0,1,2,\cdots.$
  \item[{\rm(iii)}] For the case $M_+<1$, there exists a curve $\mathcal{M}^-\subset
  \mathcal{M}^+$
  such that if
  $
   (u_-,\theta_-)\in \mathcal{M}^+,
  $
  then there exists a unique smooth solution $(\tilde{\rho},\tilde{u},\tilde{\theta})$
to the problem \eqref{BL}-\eqref{BL2}
  satisfying \eqref{decay1}.
\end{list}
\end{lemma}

Our main result is now stated as follows.
\begin{theorem} \label{thm}
  Assume that there exists an
  stationary solution $(\tilde{\rho},\tilde{u},\tilde{\theta})$ to the problem
  \eqref{BL}-\eqref{BL2}
  satisfying \eqref{decay1}.
  If
  \begin{equation}\label{thm_1}
    \inf_{x\in\mathbb{R}_+}(\rho_0,\theta_0)(x)>0,\quad
    (\rho_0-\tilde{\rho},u_0-\tilde{u},\theta_0-\tilde{\theta})\in H^1(\mathbb{R}_+),
  \end{equation}
  then there is a positive constant $\epsilon_1$ such that if $\delta\leq \epsilon_1$,
  then the outflow problem $\eqref{NS_E}$-\eqref{bdy} admits a unique solution
  $(\rho,u,\theta)$ satisfying
  \begin{equation} \label{thm_2}
      \begin{gathered}
    (\rho-\tilde{\rho},u-\tilde{u},\theta-\tilde{\theta}) \in C([0,\infty);H^1(\mathbb{R}_+)),\\
    \rho_x-\tilde{\rho}_x\in L^2(0,\infty;L^2(\mathbb{R}_+)),\quad
    (u_x-\tilde{u}_x,\theta_x-\tilde{\theta}_x)\in L^2(0,\infty;H^1(\mathbb{R}_+)),
    \end{gathered}
  \end{equation}
  and
  \begin{equation}\label{thm_3}
    \lim_{t\to\infty}\sup_{x\in\mathbb{R}_+}
    |(\rho-\tilde{\rho},u-\tilde{u},\theta-\tilde{\theta})(t,x)|=0.
  \end{equation}
\end{theorem}

\begin{remark}
  Theorem \ref{thm} shows that the non-degenerate stationary solution
  (i.e. the one which decays
  exponentially as \eqref{decay1})
  is asymptotically stable even for large initial perturbation
  with general adiabatic exponent $\gamma$, provided
that the boundary strength is small.
\end{remark}

To derive the large-time behavior of solutions toward stationary solutions,
it is sufficient to deduce certain uniform
(with respect to the time $t$) \emph{a priori} estimates on the perturbations.
The essential step is to get the lower and upper positive bounds for
both the density $\rho$ and the temperature $\theta$ uniformly in time and space.
In case of small initial perturbation, as in \cite{KNNZMR2755498},
we can use the smallness of
the \emph{a priori} $H^1$-norm of the perturbation to
obtain the uniform bounds of the density $\rho$ and the temperature $\theta$.
Based on  such uniform bounds and the smallness of the boundary strength $\delta$,
one can derive certain uniform \textmd{a priori} $H^1$-norm energy type estimates on
the perturbation.
In case that the adiabatic exponent $\gamma$ is close to 1,
as in \cite{Q11MR2785974}, we can control the lower and upper bounds of
$\theta$, and by using the Kanel's technique, we can  then
obtain the uniform positive bounds for $\rho$ provided
the boundary strength $\delta$ is small.

We are interested in showing the stability of the non-degenerate stationary solution
under large initial perturbation with general $\gamma$. 
For this purpose,
we first deduce the basic energy estimate
with the aid of the decay property of the non-degenerate stationary solution
 provided that
the boundary strength $\delta$ multiplied with a certain function of the
\emph{a priori} lower and upper bounds of density $\rho$ and temperature $\theta$
is suitably small (see Lemma \ref{lem_bas}).
Next, to get uniform pointwise bounds of the density $\rho$,
 we transform the outflow problem \eqref{NS_E}-\eqref{bdy} into
a free boundary problem in the Lagrangian coordinate.
By modifying
Jiang's argument in \cite{J99MR1671920,J02MR1912419} for fixed domains,
we will use a cut-off function with parameter to
localize the free boundary problem, and then
we will deduce a local representation of the specific volume $v=1/\rho$ to
establish the uniform bounds of $v$.
With such uniform bounds of the density $\rho$ in hand,
we can derive the $H^1$-norm (in the spatial variable $x$) estimate of
the perturbation $\theta(t,x)-\tilde{\theta}(x)$ uniformly in the time $t$
in the Eulerian coordinate.
Then the uniform upper bound of the temperature $\theta$ is obtained in light of
the Cauchy's inequality.
We note that
our derivation of the uniform upper bound of $\theta$
is motivated by the recent work \cite{LL}, which is concerned with
the stability of the constant state to the compressible Navier-Stokes equations
on unbounded but fixed domains in the Lagrangian coordinate.
The maximum principle gives us the lower bound of the temperature
$\theta$ locally in time $t$.
In view of the \emph{a priori} assumption \eqref{lem1},
we have to get the uniform positive lower bound of the temperature $\theta$,
which will be achieved by combining the local lower bound
of $\theta$ and the detailed continuation argument.

Another interesting problem is on the asymptotical stability of the
rarefaction wave and its superposition with the non-degenerate
 stationary solution to the outflow problem \eqref{NS_E}-\eqref{bdy}
under large initial perturbation.
See \cite{HQ09MR2514736,KZ08MR2420517,KZ09MR2533925} for
the isentropic case and \cite{Q11MR2785974} for
the case of small initial perturbation.
We will consider this problem in a forthcoming paper.

The layout of this paper is as follows. After stating the notations,
in section 2, we establish the desired energy estimates and
obtain the uniform bounds of both density and temperature.
In section 3, we
extend the local solution step by step to a global one
and prove the stability of the stationary solution
by combining the a priori estimates obtained in section 2 with the
the continuation argument.

\bigbreak
\noindent \emph{Notation.}
We employ $C$ or $C_i$ $(i\in\mathbb{N})$ to
denote a generic positive constant which is
independent of $t$, $x$ and $\delta$.
Notice that all the constants may change from line to line.
The Gaussian bracket $[x]$ means
the largest integer not greater than $x$.
For function spaces, $L^p(\mathbb{R}_+)$~$(1\leq p\leq \infty)$ stands for
the usual Lebesgue space on $\mathbb{R}_+$ equipped with the norm $\|{\cdot}\|_{L^p}$
and~$H^k(\mathbb{R}_+)$~the usual Sobolev space in the $L^2$ sense with
norm~$\|\cdot\|_k$.
We use the notation $\|\cdot\|=\|\cdot\|_{L^2}$.
We denote by $C^k(I; H^p)$ the space of $k$-times continuously
differentiable functions on the interval $I$ with values in
$H^p(\mathbb{R}_+)$ and~$L^2(I; H^p)$ the space of~$L^2$-functions
on $I$ with values in~$H^p(\mathbb{R}_+)$.

\section{A priori estimates}
\subsection{Reformation of the problem}
This section is devoted to deriving a priori estimates on the solution
$(\rho,u,\theta)$ to the outflow problem \eqref{NS_E}-\eqref{bdy}.
To this end, we regard the solution $(\rho,u,\theta)$ as a perturbation
from the stationary solution $(\tilde{\rho},\tilde{u},\tilde{\theta})$
and put the perturbation $(\phi,\psi,\vartheta)$ by
\begin{equation}\label{per_def}
  (\phi, \psi, \vartheta)(t,x):=(\rho, u, \theta)(t,x)-(\tilde{\rho}, \tilde{u}, \tilde{\theta})(x).
\end{equation}
Subtracting \eqref{BL}-\eqref{BL2} from \eqref{NS_E}-\eqref{bdy} yields
\begin{equation}\label{per}
\left\{
  \begin{aligned}
    \phi_t+u\phi_x+\rho\psi_x&=-\tilde{u}_x\phi-\tilde{\rho}_x\psi,\qquad\quad t>0,\ x\in\mathbb{R}_{+},\\[0.5mm]
    \rho(\psi_t+u\psi_x)+(P-\tilde{P})_x&=\mu\psi_{xx}+g,\\[0.5mm]
    c_v\rho(\vartheta_t+u\vartheta_x)+P\psi_x
    &=\kappa\vartheta_{xx}+\mu\psi_x^2+h,\\[0.5mm]
    (\phi, \psi, \vartheta)|_{t=0}&=(\phi_0, \psi_0, \vartheta_0),\\[0.5mm]
    (\psi, \vartheta)|_{x=0}&=(0,0),
  \end{aligned}\right.
\end{equation}
where
\begin{equation} \label{gh}
     g:=-\tilde{u}_x(\tilde{u}\phi+{\rho}\psi),\quad
    h:=-c_v\tilde{\theta}_x(\tilde{u}\phi+{\rho}\psi)
    -\tilde{u}_x(P-\tilde{P})
    +2\mu\tilde{u}_x\psi_x,
\end{equation}
and the initial condition
$(\phi_0, \psi_0, \vartheta_0)
    :=(\rho_0-\tilde{\rho},u_0-\tilde{u},\theta_0-\tilde{\theta})$
    satisfies
\begin{equation}\label{per_far}
  (\phi_0, \psi_0, \vartheta_0)(x)
    \to(0,0,0),\quad {\rm as}\ x\to \infty.
\end{equation}

The solution space $X(0,T;m_1,M_1;m_2,M_2)$ is defined by
\begin{multline*}
       X(0,T;m_1,M_1;m_2,M_2):=\big\{(\phi,\psi,\vartheta)
      \in C([0,T];H^1):
     (\psi_x,\vartheta_x)\in L^2(0,T;H^1),\\
     \phi_x\in L^2(0,T;L^2),\
     m_1\leq  \phi+\tilde{\rho}\leq M_1,\ m_2\leq
     \vartheta+\tilde{\theta}\leq M_2\big\}.
\end{multline*}
We summarize the local existence of solutions
to the problem \eqref{per} in the following proposition,
which can be proved by the standard iteration method (see \cite{HMS04MR2040072}).
\begin{proposition}[Local existence] \label{Pro_loc}
   Suppose that the conditions in Theorem \ref{thm} hold.
   If  $\|(\phi_0,\psi_0,\vartheta_0)\|_1\leq M,$
  $\lambda_1\leq \phi_0(x)+\tilde{\rho}(x)\leq \Lambda_1,$ and
  $\lambda_2\leq \vartheta_0(x)+\tilde{\theta}(x)\leq \Lambda_2$ hold
for all $x\in\mathbb{R}_+$,
then  \eqref{per} admits a unique solution $(\phi,\psi,\vartheta)\in
X\left(0,T_0;\frac{1}{2}\lambda_1,2\Lambda_1;
\frac{1}{2}\lambda_2,2\Lambda_2\right)$
for some constant $T_0=T_0(\lambda_{1},\lambda_{2},M)>0$
depending only on $\lambda_{1}$, $\lambda_{2}$ and $M$.
\end{proposition}

We now turn to deduce certain uniform a priori estimates for the perturbation
$(\phi,\psi,\vartheta)\in X(0,T;m_1,M_1,m_2,M_2)$ in the Sobolev space $H^1$.
Before doing so, we recall that
$C$ or $C_i$ $(i\in\mathbb{N})$ will be used to denote
some generic positive constant which depends only on
$\inf_{\mathbb{R}_+}(\rho_0,\theta_0)$
  and $\|(\phi_0,\psi_0,\vartheta_0)\|_1$.
  For notational simplicity, we introduce
  $A\lesssim B$ if $A\leq C B$ holds uniformly for some constant
  $C$ depending only on
 $\inf_{\mathbb{R}_+}(\rho_0,\theta_0)$
  and $\|(\phi_0,\psi_0,\vartheta_0)\|_1$.
  The notation $A\sim B$ means that both $A\lesssim B$  and $B\lesssim A$.
  Besides, we will use the notation $(\rho, \theta)=(\phi+\tilde{\rho},
  \vartheta+\tilde{\theta})$ so that
  \begin{equation}\label{apriori}
    m_1\leq {\rho}(t,x)\leq M_1,\quad
  m_2\leq   {\theta}(t,x)\leq M_2, \quad {\rm for\ all}\
  (t,x)\in[0,T]\times\mathbb{R}_+.
  \end{equation}
Without loss of generality, we may assume that $m_i\leq 1\leq M_i$ for $i=1,2$.

\subsection{Basic energy estimate}

First, we have the following basic energy estimate.
\begin{lemma}\label{lem_bas}
 There exists a sufficiently small $\epsilon_0>0$ such that if
 \begin{equation}\label{lem1}
   \Xi(m_1,M_1,m_2,M_2)\delta\leq \epsilon_0,\quad
   \Xi(m_1,M_1,m_2,M_2):=m_1^{-10}M_1^{10}m_2^{-10}M_2^{10},
 \end{equation}
 then
 \begin{align}\label{lem2}
   &\int_{\mathbb{R}_+}\frac{\phi_x^2}{\rho^3}+\int_0^t\frac{\phi_x^2}{\rho^3}(s,0)ds
    +\int_0^t\int_{\mathbb{R}_+}\frac{\theta\phi_x^2}{\rho^2}
    \lesssim 1+\|\theta\|_{L^\infty([0,T]\times\mathbb{R}_+)},
  \\ \label{lem3}
    &\int_{\mathbb{R}_+}\rho \mathcal{E}
    +\int_0^t\rho\Phi\left(\frac{\tilde{\rho}}{\rho}\right)(s,0)ds
    +\int_0^t\int_{\mathbb{R}_+}\left[\frac{\psi_x^2}{\theta}+
    \frac{\vartheta_x^2}{\theta^2}\right]
    \lesssim 1,
  \end{align}
  where
  \begin{equation}\label{lem2.1.4}
      \mathcal{E}:=R\tilde{\theta}\Phi\left(\frac{\tilde{\rho}}{\rho}\right)+\frac12\psi^2
  +c_v\tilde{\theta}\Phi\left(\frac{\theta}{\tilde{\theta}}\right),
  \quad \Phi(z):= z-\ln z-1.
  \end{equation}
\end{lemma}
\begin{proof}
By a straightforward calculation, we find
  \begin{equation}\label{pro2.1.1}
  (\rho \mathcal{E})_t+\left[\rho u\mathcal{E}+\psi(P-\tilde{P})-\mu\psi\psi_x
  -\kappa\frac{\vartheta\vartheta_x}{\theta}\right]_x
  +\mu\frac{\tilde{\theta}\psi_x^2}{\theta}
  +\kappa\frac{\tilde{\theta}\vartheta_x^2}{\theta^2}=\mathcal{R},
  \end{equation}
  where
  \begin{equation*}\label{R}
    \begin{aligned}
    \mathcal{R}=~&\rho\tilde{\theta}_x\left[R\tilde{u}\Phi\left(\frac{\tilde{\rho}}{\rho}\right)
    +c_v\tilde{u}\Phi\left(\frac{\theta}{\tilde{\theta}}\right)
    -c_v\psi\ln\left(\frac{\theta}{\tilde{\theta}}\right)
    -R\psi\ln\left(\frac{\tilde{\rho}}{\rho}\right)\right]
    -\rho\tilde{u}_x\psi^2\\[2mm]
    &-\left(c_v\tilde{\rho}\tilde{u}\tilde{\theta}_x
    +\tilde{P}\tilde{u}_x\right)\left(\frac{\phi\vartheta}{\tilde{\rho}\tilde{\theta}}
    +\frac{\vartheta^2}{\tilde{\theta}\theta}\right)
    +2\mu\frac{\tilde{u}_x\psi_x\vartheta}{\theta}
    +\kappa\frac{\tilde{\theta}_x\vartheta\vartheta_x}{\theta^2}
    -\mu\frac{\tilde{u}_{xx}\phi\psi}{\tilde{\rho}}.
    \end{aligned}
  \end{equation*}
The identities
\begin{equation*}
    \Phi(z)=\int_0^1\int_0^1\theta_1\Phi''(1+\theta_1\theta_2(z-1))d\theta_2d\theta_1
    (z-1)^2
\end{equation*}
and
\begin{equation*}
    \ln z=\int_0^1\frac{(z-1)d\theta_1}{1+\theta_1(z-1)}
\end{equation*}
imply
\begin{equation}\label{pro2.1.2}
  \Phi(z)+|\ln z|^2\lesssim (z^{-1}+1)^2(z-1)^2,
  \quad (z+1)^{-2}(z-1)^2\lesssim \Phi(z).
\end{equation}
In view of \eqref{apriori} and \eqref{pro2.1.2},
we apply Cauchy's inequality to $\mathcal{R}$ and obtain
\begin{equation*}
  \mathcal{R}
  \lesssim \left|\left(\tilde{\rho}_x,\tilde{u}_x,\tilde{\theta}_x,\tilde{u}_{xx}\right)\right|
  \left[M_1m_1^{-2}\phi^2+M_1m_2^{-2}\vartheta^2+M_1\psi^2+
  \frac{\psi_x^2}{\theta}+\frac{\vartheta_x^2}{\theta^2}\right].
\end{equation*}
Integrate \eqref{pro2.1.1} over $[0,T]\times\mathbb{R}_+$
and use the good sign of $u_-$ to derive
\begin{equation}\label{pro2.1.3}
\begin{aligned}
  \int_{\mathbb{R}_+}&\rho \mathcal{E}
  +\int_0^t\rho\Phi\left(\frac{\tilde{\rho}}{\rho}\right)
  (s,0)ds+\int_0^t\int_{\mathbb{R}_+}\left[\frac{\psi_x^2}{\theta}
  +\frac{\vartheta_x^2}{\theta^2}\right]\\
  \lesssim~&1+\int_0^t\int_{\mathbb{R}_+}
  \left|\left(\tilde{\rho}_x,\tilde{u}_x,\tilde{\theta}_x,\tilde{u}_{xx}\right)\right|
  \left[M_1m_1^{-2}\phi^2+M_1m_2^{-2}\vartheta^2+M_1\psi^2+
  \frac{\psi_x^2}{\theta}+\frac{\vartheta_x^2}{\theta^2}\right].
\end{aligned}
\end{equation}
To estimate the terms on the right-hand side of \eqref{pro2.1.3},
we use an idea in Nikkuni-Kawashima \cite{NK00MR1793167}, i.e.
a Poincar\'{e} type inequality
$$
|\varphi(t,x)|\leq|\varphi(t,0)|+\sqrt{x}\|\varphi_x(t)\|,\ \ x\in\mathbb{R}_+.
$$
Applying this inequality to $\phi$, $\psi$ and $\vartheta$, we deduce
from  \eqref{decay1} and \eqref{pro2.1.2} that
\begin{equation}\label{pro2.1.4}
\begin{split}
   \sum_{l=0,1}\int_0^t\int_{\mathbb{R}_+}&
  \left|\partial_x^k(\tilde{\rho},\tilde{u},\tilde{\theta})\right|
  \left|\partial_x^l(\phi,\psi,\vartheta)\right|^2\\
  \lesssim~& \delta\int_0^t\left[|\phi(s,0)|^2
  +\left\|\left(\phi_x,\psi_x,\vartheta_x\right)(s)\right\|^2\right] ds\\
  \lesssim~&
  M_1^2m_1^{-1}\delta\int_0^t\rho\Phi\left(\frac{\tilde{\rho}}{\rho}\right)
  (s,0)ds\\
  &+\delta\int_0^t\|\phi_x(s)\|^2ds
  +M_2^2\delta\int_0^t\int_{\mathbb{R}_+}\left[\frac{\psi_x^2}{\theta}
  +\frac{\vartheta_x^2}{\theta^2}\right]
\end{split}
\end{equation}
for each integer $k=1,2,\cdots.$
Plug \eqref{pro2.1.4} into \eqref{pro2.1.3}
to derive
\begin{equation*}
  \begin{split}
      \int_{\mathbb{R}_+}&\rho \mathcal{E}
  +\int_0^t\rho\Phi\left(\frac{\tilde{\rho}}{\rho}\right)
  (s,0)ds+\int_0^t\int_{\mathbb{R}_+}\left[\frac{\psi_x^2}{\theta}
  +\frac{\vartheta_x^2}{\theta^2}\right]\\
     \lesssim~ &
     1+m_1^{-3}M_1^{3}m_2^{-2}\delta
     \int_0^t\rho\Phi\left(\frac{\tilde{\rho}}{\rho}\right)
  (s,0)ds\\
     &+m_1^{-2}M_1m_2^{-2}\delta\int_0^t\|\phi_x(s)\|^2ds
     +m_1^{-2}M_1m_2^{-2}M_2^2
     \int_0^t\int_{\mathbb{R}_+}\left[\frac{\psi_x^2}{\theta}
  +\frac{\vartheta_x^2}{\theta^2}\right],
  \end{split}
\end{equation*}
which implies
\begin{equation}\label{pro2.1.5}
\begin{split}
  \int_{\mathbb{R}_+}&\rho \mathcal{E}
  +\int_0^t\rho\Phi\left(\frac{\tilde{\rho}}{\rho}\right)
  (s,0)ds+\int_0^t\int_{\mathbb{R}_+}\left[\frac{\psi_x^2}{\theta}
  +\frac{\vartheta_x^2}{\theta^2}\right]\\
  \lesssim~&
  1+m_1^{-2}M_1 m_2^{-2}\delta\int_0^t\|\phi_x(s)\|^2ds,
\end{split}
\end{equation}
provided \eqref{lem1} holds for a suitable small $\epsilon_0>0$.

To control the last term of \eqref{pro2.1.5},
we differentiate $\eqref{per}_1$ (first equation of \eqref{per})
with respect to $x$, and then multiply the resulting identity and
$\eqref{per}_2$ by $\frac{\phi_x}{\rho^3}$ and $\frac{\phi_x}{\mu\rho^2}$,
respectively, to discover
\begin{equation}\label{pro2.1.6}
\left[\frac{\phi_x^2}{2\rho^3}+\frac{\phi_x\psi}{\mu\rho}\right]_t
  +\left[\frac{u\phi_x^2}{2\rho^3}-\frac{\psi\phi_t}{\mu\rho}
  -\frac{\tilde{\rho}_x\psi^2}{\mu\rho}\right]_x+\frac{R\theta\phi_x^2}{\mu\rho^2}
  =Q
\end{equation}
with
\begin{equation*}
  \begin{split}
     Q=~&\frac{\psi_x^2}{\mu}-\frac{R\phi_x\vartheta_x}{\mu\rho}
  -\left[\frac{\tilde{u}\tilde{u}_x}{\mu\rho^2}
  +\frac{R\tilde{\theta}_x}{\mu\rho^2}+\frac{\tilde{u}_{xx}}{\rho^3}\right]\phi\phi_x
  -\left[\frac{2\tilde{\rho}_x\psi_x}{\rho^3}
  +\frac{R\tilde{\rho}_x\vartheta}{\mu\rho^2}\right]\phi_x\\
  &-\left[\frac{2\tilde{\rho}_x\psi}{\mu\rho}-\frac{\tilde{u}_x\phi}
  {\mu\rho}\right]\psi_x
  -\left[\frac{\tilde{u}_x}{\mu\rho}
  -\frac{\tilde{\rho}\tilde{u}_x}{\mu\rho^2}+\frac{\tilde{\rho}_{xx}}{\rho^3}\right]
  \psi\phi_x
  -\frac{\tilde{\rho}_x\tilde{u}_x\phi\psi}{\mu\rho^2}
  -\frac{\tilde{\rho}_{xx}\psi^2}{\mu\rho}.
  \end{split}
\end{equation*}
We integrate \eqref{pro2.1.6} over $[0,t]\times\mathbb{R}_+$ and use
Cauchy's inequality
to have
\begin{equation}\label{pro2.1'}
  \int_{\mathbb{R}_+}\frac{\phi_x^2}{\rho^3}
+\int_0^t\frac{\phi_x^2}{\rho^3}(s,0)
+\int_0^t\int_{\mathbb{R}_+}\frac{\theta\phi_x^2}{\rho^2}
\lesssim
1+\int_{\mathbb{R}_+}\rho\psi^2
+\int_0^t\int_{\mathbb{R}_+}|Q|.
\end{equation}
Cauchy's inequality yields the bound
\begin{equation*}
  |Q|\lesssim
  \psi_x^2+C(\epsilon)\frac{\vartheta_x^2}{\theta}+
  \epsilon\frac{\theta\phi_x^2}{\rho^2}
  +m_1^{-3}
  | (\tilde{\rho}_x,\tilde{u}_x,\tilde{\theta}_x,
  \tilde{\rho}_{xx},\tilde{u}_{xx})|
  |(\phi,\psi,\vartheta,\phi_x,\psi_x)|^2.
\end{equation*}
Plugging this inequality into \eqref{pro2.1'},
we take $\epsilon>0$ suitable small  and utilize
\eqref{pro2.1.4} and
\eqref{pro2.1.5} to obtain
\begin{equation*}\label{pro2.1.7}
\begin{aligned}
\int_{\mathbb{R}_+}&\frac{\phi_x^2}{\rho^3}
+\int_0^t\frac{\phi_x^2}{\rho^3}(s,0)
+\int_0^t\int_{\mathbb{R}_+}\frac{\theta\phi_x^2}{\rho^2}\\
\lesssim~&1
+\int_0^t\int_{\mathbb{R}_+}\left[
\psi_x^2+\frac{\vartheta_x^2}{\theta}\right]
+m_1^{-3}M_1^3m_2^{-2}M_2^2\delta\int_0^t\|\phi_x(s)\|^2ds\\
&+m_1^{-4}M_1^2M_2^2\delta
\left[\int_0^t\rho\Phi\left(\frac{\tilde{\rho}}{\rho}\right)
  (s,0)ds+\int_0^t\int_{\mathbb{R}_+}\left(\frac{\psi_x^2}{\theta}
  +\frac{\vartheta_x^2}{\theta^2}\right)\right]\\
\lesssim~&1
+\int_0^t\int_{\mathbb{R}_+}\left[
\psi_x^2+\frac{\vartheta_x^2}{\theta}\right]
+
 m_1^{-10}M_1^{10}m_2^{-10}M_2^{10}\delta \int_0^t\int_{\mathbb{R}_+}\frac{\theta\phi_x^2}{\rho^2}.
\end{aligned}
\end{equation*}
Taking  $\epsilon_0>0$ small enough,
we can derive from \eqref{lem1} that
\begin{equation}\label{pro2.1.8}
  \int_{\mathbb{R}_+}\frac{\phi_x^2}{\rho^3}
+\int_0^t\frac{\phi_x^2}{\rho^3}(s,0)
+\int_0^t\int_{\mathbb{R}_+}\frac{\theta\phi_x^2}{\rho^2}
\lesssim 1+\int_0^t\int_{\mathbb{R}_+}\left[
\psi_x^2+\frac{\vartheta_x^2}{\theta}\right].
\end{equation}
The estimate \eqref{lem2}
follows by
plugging \eqref{pro2.1.5} into \eqref{pro2.1.8}
under the condition \eqref{lem1} for a sufficiently small $\epsilon_0>0$.
We plug \eqref{lem2} into \eqref{pro2.1.5} to deduce the estimate \eqref{lem3}.
The proof of the lemma is completed.
\end{proof}

\subsection{Pointwise estimates of density}
Once the basic energy estimate \eqref{lem3} is obtained,
we can proceed to deduce the positive lower and upper bounds
of the density uniformly in time.
For this purpose, it is easier to consider the problem
in the Lagrangian coordinate than in the Eulerian coordinate.
We introduce the Lagrangian variable
\begin{equation}\label{y}
  y=-u_-\int_{0}^{t}\rho(s,0)ds+\int_{0}^{x}\rho(t,z)dz,
\end{equation}
and define
$(\hat{\rho},\hat{u},\hat{\theta})(t,y):=({\rho},{u},{\theta})(t,x)$.
By the coordinate change $(t,x)\mapsto(t,y)$, the domain
$[0,T]\times\mathbb{R}_+$ is mapped into
$$\Omega_T:=\{(t,y):0\leq t\leq T, y>Y(t)\},$$ where
$$Y(t):=-u_-\int_{0}^{t}\rho(s,0)ds.$$
The outflow problem \eqref{NS_E}-\eqref{bdy} can be transformed into
the problem in the Lagrangian coordinate:
\begin{equation}\label{Lag}
\left\{
  \begin{aligned}
    v_{t}-u_y&=0,\qquad\qquad\qquad\qquad y>Y(t),\\[2mm]
    u_{t}+P_y&=\left(\frac{\mu u_y}{v}\right)_y, \\
    \left(c_v\theta+\frac{u^2}{2}\right)_{t}+(Pu)_y&=\left(\frac{\kappa\theta_y}{v}
    +\frac{\mu uu_y}{v}\right)_y,\\
    (u,\theta)|_{y=Y(t)}&=(u_-,\theta_-),\\[3mm]
    (v,u,\theta)|_{t=0}&=(v_0,u_0,\theta_0),
  \end{aligned}\right.
\end{equation}
where $v={1}/{\rho}$ stands for the specific volume of the gas.
We drop the hats in the formula in this subsection for
simplicity of notation.
The basic energy estimate \eqref{lem3} in Eulerian coordinate
can be transformed into a corresponding estimate in Lagrangian coordinate
as a corollary of Lemma \ref{lem_bas}.
\begin{corollary}
If \eqref{lem1} holds for a sufficiently small $\epsilon_0>0$, then
\begin{equation}\label{est_El}
  \sup_{0\leq t\leq T}\int_{Y(t)}^{\infty}{\mathcal{E}} dy
    +\int_0^{T}\int_{Y(t)}^{\infty}\left[\frac{\psi_y^2}{v \theta}
  +\frac{\vartheta_y^2}{v\theta^2}\right]
  \lesssim 1.
\end{equation}
\end{corollary}
Note that the function $Y(t)$ describing the boundary in the Lagrangian coordinate
is part of the unknown, i.e. the problem \eqref{Lag} is a free boundary problem.
To obtain the uniform bounds of the specific volume $v$
for the free boundary problem \eqref{Lag},
we introduce
\begin{equation} \label{Omegai}
  \Omega_{i}(t):=
  \begin{cases}
    [Y(t),[Y(t)]+2],\quad & i=[Y(t)]+1,\\
    [i,i+1],\quad &i\geq[Y(t)]+2
  \end{cases}
\end{equation}
for any integer $i\geq[Y(t)]+1$.
Based on  the basic energy estimate \eqref{est_El}, we have the following lemma.
 \begin{lemma} \label{lem_v1}
 Suppose that
 \eqref{lem1} holds for a sufficiently small $\epsilon_0>0$.
 Then there exists a constant $C_0>0$, which depends only on
   $\inf_{\mathbb{R}_+}(\rho_0,\theta_0)$
  and $\|(\phi_0,\psi_0,\vartheta_0)\|_1$, such that
 for
 all pair $(s,t)$ with $0\leq s\leq t\leq T$
and  integer $i\geq[Y(t)]+1$,
 \begin{equation}\label{bound1}
   C_0^{-1}\leq \int_{\Omega_i(t)}{v}(s,y)dy,
    \int_{\Omega_i(t)}{\theta}(s,y)dy\leq C_0,
 \end{equation}
 and there are points $a_i(s,t),b_i(s,t)\in\Omega_i(t)$
    such that
    \begin{gather} \label{bound2}
       C_0^{-1}\leq v(s,a_i(s,t)) , \theta(s,b_i(s,t))\leq C_0.
   \end{gather}
  \end{lemma}
\begin{proof}
  Let $0\leq s\leq t\leq T$ and $i\geq[Y(t)]+1$.
  From the definition of $Y(t)$ and the sign of $u_-$,
  we have $Y(s)\leq Y(t)$ and $\Omega_{i}(t)\subset [Y(s),\infty).$
  In view of \eqref{est_El}, we get
  \begin{equation*}
    \int_{\Omega_i(t)}\Phi\left(\frac{v}{\tilde{v}}\right)(s,y)dy+
    \int_{\Omega_i(t)}\Phi\left(\frac{\theta}{\tilde{\theta}}\right)(s,y)dy
    \lesssim 1.
  \end{equation*}
  Apply Jensen's inequality to
the convex function $\Phi$ to obtain
  \begin{equation*}
     \Phi\left(\frac{1}{|\Omega_i(t)|}\int_{\Omega_i(t)}\frac{v}{\tilde{v}}(s,y)dy\right)+
    \Phi\left(\frac{1}{|\Omega_i(t)|}\int_{\Omega_i(t)}\frac{\theta}{\tilde{\theta}}
    (s,y)dy \right)\leq C.
  \end{equation*}
  Let $\alpha$ and $\beta$ be the two positive roots of the equation $\Phi(z)=C$.
  Then we have
  \begin{equation*}
    \alpha\leq
    \frac{1}{|\Omega_i(t)|}\int_{\Omega_i(t)}\frac{v}{\tilde{v}}(s,y)dy,\
    \frac{1}{|\Omega_i(t)|}\int_{\Omega_i(t)}\frac{\theta}{\tilde{\theta}}
    (s,y)dy
    \leq \beta.
  \end{equation*}
  These estimates imply \eqref{bound1}.
  Finally we employ the mean value theorem
 to \eqref{bound1} to find $a_i(s,t),b_i(s,t)\in\Omega_i(t)$ satisfying
  \eqref{bound2}.
  The proof of the lemma is completed.
\end{proof}
We deduce a local representation of $v$ in the next lemma
by modifying Jiang's argument in  \cite{J99MR1671920,J02MR1912419} for
fixed domains.
To this end, we introduce the cutoff function
  $\varphi_z\in W^{1,\infty}(\mathbb{R})$ with
parameter $z\in\mathbb{R}$ by
  \begin{equation} \label{varphi}
  \varphi_z(y)=
  \begin{cases}
    1,\quad& y<[z]+4,\\
    [z]+5-y,\quad& [z]+4\leq y< [z]+5,\\
    0,\quad& y\geq [z]+5.
  \end{cases}
\end{equation}
\begin{lemma}
  Let $(\tau,z)\in\Omega_T$. Then we have
\begin{equation}\label{v_form}
  v(t,y)=B_z(t,y)A_z(t)+\frac{R}{\mu}\int_0^t
  \frac{B_z(t,y)A_z(t)}{B_z(s,y)A_z(s)}\theta(s,y)ds
\end{equation}
for all $t\in[0,\tau]$ and $y\in
I_z(\tau):=(Y(\tau),\infty)\cap([z]-1,[z]+4)$, where
  \begin{align}   \label{B}
    B_z(t,y)&:=v_0(y)\exp\left\{\frac{1}{\mu}
    \int_{y}^{\infty}
        \left(u_0(\xi)-u(t,\xi)\right)\varphi_z(\xi)d\xi\right\},\\   \label{Y}
    A_z(t)&:=\exp\left\{\frac{1}{\mu}\int_0^t
    \int_{[z]+4}^{[z]+5}\left(\frac{\mu u_y}{v}-P\right)d\xi ds\right\}.
  \end{align}
\end{lemma}
\begin{proof}
We  multiply $\eqref{Lag}_2$ by $\varphi_z$ to get
 \begin{equation} \label{ident1}
   (\varphi_z u)_t=\left[\left(\mu\frac{u_y}{v}-P\right)\varphi_z\right]_y
   -\varphi_z'\left(\mu\frac{u_y}{v}-P\right).
 \end{equation}
Let $(t,y)\in[0,\tau]\times I_z(\tau)$.
Since  $y> 
Y(s)$ for each $s\in [0,\tau]$, we have
$[0,\tau]\times[y,\infty)\subset\Omega_T$.
In view of the identity $\varphi_z(y)=1$ and \eqref{Lag}$_1$,
we integrate \eqref{ident1} over
$[0,t]\times[y,\infty)$ to get
 \begin{equation*}
   -\int_y^{\infty}\varphi_z(\xi)(u(t,\xi)-u_0(\xi))d\xi
   =\mu\ln\frac{v(t,y)}{v_0(y)}-R\int_0^{t}\frac{\theta(s,y)}{v(s,y)}ds
   +\int_0^{t}\int_{[z]+4}^{[z]+5}\left(P-\mu\frac{u_y}{v}\right).
 \end{equation*}
 This implies that for each $t\in[0,\tau]$,
 \begin{equation}\label{pro4}
  \frac{1}{v(t,y)}\exp\left\{\frac{R}{\mu}\int_0^t\frac{\theta(s,y)}{v(s,y)}ds\right\}
  =\frac{1}{B_z(t,y)A_z(t)}.
\end{equation}
Multiplying \eqref{pro4} by $R\theta(t,y)/\mu$
and integrating the resulting identity over $[0, t]$, we have
\begin{equation*}
  \exp\left\{\frac{R}{\mu}\int_0^{t}\frac{\theta(s,y)}{v(s,y)}ds\right\}
  =1+\frac{R}{\mu}\int_0^{t} \frac{\theta(s,y)}{B_z(s,y)A_z(s)}ds.
\end{equation*}
We then plug this identity into \eqref{pro4} to obtain
\eqref{v_form} and complete the proof of the lemma.
\end{proof}
The following lemma
is devoted to showing
 the bounds of the specific volume $v(\tau,z)$
 uniformly in the time $\tau$ and the Lagrangian variable $z$.
\begin{lemma} \label{lem_boun}
   If \eqref{lem1} holds for a sufficiently small $\epsilon_0>0$,
   then
     \begin{equation}\label{boun_v}
   C_1^{-1}\leq v(\tau,z)\leq C_1\quad {\rm for\ all}\ (\tau,z)\in\Omega_T.
  \end{equation}
\end{lemma}
\begin{proof}
Let $(\tau,z)\in\Omega_T$ be arbitrary but fixed.
The proof is divided into three steps.

\noindent{\bf Step 1}.
In view of Cauchy's inequality and \eqref{est_El}, we have
\begin{equation}\label{est_B}
  B_z(t,y)\sim1 \quad {\rm for\ all}\ (t,y)\in[0,\tau]\times I_z(\tau).
\end{equation}
Let $0\leq s\leq t\leq \tau$.
Apply
Cauchy's inequality and Jensen's inequality
for the convex function $1/x$ $(x>0)$,
\eqref{decay1}, \eqref{est_El} and \eqref{bound1} to deduce
\begin{equation}\label{pro5}
    \begin{aligned}
      &\int_s^{t}
    \int_{[z]+4}^{[z]+5}\left[\frac{\mu u_y}{v}-P\right]\\
        & \leq C \int_s^{t}\int_{[z]+4}^{[z]+5}\frac{u_y^2}{v\theta}
        +\frac R2\int_s^{t}\int_{[z]+4}^{[z]+5}\frac{\theta}{v}
        -\int_s^t\int_{[z]+4}^{[z]+5} \frac{R\theta}{v} \\[1.5mm]
        &\leq C \int_s^{t}\int_{[z]+4}^{[z]+5}\left[\frac{\psi_y^2}{v\theta}
        +\frac{\tilde{u}_y^2}{v\theta}\right]
        -\frac R2\int_s^{t}\int_{[z]+4}^{[z]+5}\frac{\theta}{v}\\
        &\leq C+C M_1   m_2^{-1}\delta ({t}-s)
        -\frac{R}{2}\int_s^{t} \inf_{([z]+4,[z]+5)}\theta(t',\cdot)
        \left[\int_{[z]+4}^{[z]+5}v\right]^{-1}d t'\\
        &\leq C+C\epsilon_0 (t-s)
        -C^{-1}\int_s^{t} \inf_{([z]+4,[z]+5)}\theta(t',\cdot)d t',
    \end{aligned}
  \end{equation}
For each $0\leq t' \leq t$, there exists $y(t')\in\left[[z]+4,[z]+5\right]$
such that
\begin{equation*} \label{pro12}
  \theta(t',y(t' ))=\inf_{([z]+4,[z]+5)}\theta(t' ,\cdot),
\end{equation*}
Since $[z]+4\geq [Y(t')]+2$, we derive
from the definition \eqref{Omegai} that
$
  \Omega_{[z]+4}(t')=\left[[z]+4,[z]+5\right].
$
We then apply H\"{o}lder's and Cauchy's inequalities to obtain from
\eqref{decay1}, Lemma \ref{lem_v1} and \eqref{est_El} that
\begin{equation}\label{pro10}
\begin{split}
  &\left|\int_s^{t}\int_{b_{[z]+4}(t',t')}^{y(t')}
\frac{\theta_y}{\theta}(t' ,\xi)d\xi dt'
\right|\\
 ~&\leq
\int_s^{t}\int_{\Omega_{[z]+4}(t')}
\left|\frac{\tilde\theta_y}{\theta}(t',\xi)+
\frac{\vartheta_y}{\theta}(t' ,\xi)\right|d\xi dt' \\
~&\leq
m_2^{-1}\delta({t}-s)
+\int_s^{t}\left|\int_{Y(t')}^{\infty}
\frac{\vartheta_y^2}{v\theta^2}(t',\xi)d\xi\right|^{\frac{1}{2}}
\left|\int_{\Omega_{[z]+4}(t')} v (t',\xi)d\xi\right|^{\frac{1}{2}}dt' \\
~&\leq C (t-s) +C \int_s^t\int_{Y(t')}^{\infty}
\frac{\vartheta_y^2}{v\theta^2}(t',\xi)d\xi dt'  \\
~&\leq C (t-s) +C.
\end{split}
\end{equation}
Applying Jensen's inequality to the convex function $e^{x}$,
we have from 
\eqref{bound2} and \eqref{pro10} that
\begin{equation*}\label{pro11}
  \begin{split}
   & \int_s^{t}\inf_{([z]+4,[z]+5)}\theta(t',\cdot)dt'
    =\int_s^{t}\exp\left(\ln\theta(t' ,y(t' ))\right)dt' \\
  ~&  \geq ({t}-s)\exp\left(\frac{1}{{t}-s}\int_s^{t}
    \ln\theta(t' ,y(t ' )) dt' \right)\\
 ~&   \geq  ({t}-s)\exp\left(\frac{1}{{t}-s}\int_s^{t}
\left[\int_{b_{[z]+4}(t ',t' )}^{y(t'  )}
\frac{\theta_y}{\theta}(t' ,\xi)d\xi+
\ln\theta(t' ,b_{[z]+4}(t' ,t '))\right]dt'\right)\\
~&\geq ({t}-s)\exp\left(-\ln C_0
-\frac{1}{{t}-s}\left|\int_s^{t}\int_{b_{[z]+4}(t ',t' )}^{y(t' )}
\frac{\theta_y}{\theta}(t ',\xi)d\xi dt'
\right|\right)\\
~&\geq \frac{{t}-s}{C}\exp\left({-\frac{C}{ {t}-s }}\right).
  \end{split}
\end{equation*}
This implies
\begin{equation}\label{pro6}
  -\int_s^{t} \inf_{([z]+4,[z]+5)}\theta(t' ,\cdot)dt '
  \leq
  \begin{cases}
    0, \quad\quad &{\rm for }\ 0\leq {t}-s\leq 1,\\
    -C^{-1}({t}-s)\quad\quad &{\rm for }\ {t}-s\geq 1.
  \end{cases}
\end{equation}
Plugging \eqref{pro6} into \eqref{pro5} and taking $\epsilon_0>0$ small enough,
we have  for each $s\in[0,{t}]$ that
\begin{equation*}
 \int_s^{t}
    \int_{[z]+4}^{[z]+5}\left[\frac{\mu u_y}{v}-P\right]
    \leq C-C^{-1}({t}-s).
\end{equation*}
According to the definition \eqref{Y}, we then obtain
\begin{equation}\label{pro7}
    0\leq A_z({t})\leq C\mathrm{e}^{-{t}/C},\quad
    \frac{A_z(t)}{A_z(s)}\leq C\mathrm{e}^{-({t}-s)/{C}}
    \qquad {\rm for\ all}\ 0\leq s\leq {t}\leq \tau.
  \end{equation}

  \noindent{\bf Step 2}.
   Plugging \eqref{est_B} and \eqref{pro7} into \eqref{v_form}, we infer
   \begin{equation}\label{pro2.2.30}
     \int_0^t\frac{A_z(t)}{A_z(s)}\theta(s,y)ds\lesssim
     v(t,y)\lesssim 1+\int_0^t\theta(s,y)\mathrm{e}^{-\frac{t-s}C}ds
     \end{equation}
for all $(t,y)\in [0,\tau]\times I_z(\tau)$.
In light of the fundamental theorem of calculus, we infer
from  \eqref{decay1} and \eqref{bound1} that for
 $y\in I_z(\tau)$ and $0\leq s\leq t\leq \tau$,
     \begin{equation} \label{pro14}
  \begin{aligned}
    ~&\left|\theta(s,y)^{\frac{1}{2}}-\theta(s,b_{[z]+2}(s,\tau))^{\frac{1}{2}}
    \right|\\
    ~&\lesssim
     \int_{I_z(\tau)}\theta^{-\frac12}|\theta_y|(s,\xi)d\xi\\
    ~&\lesssim \int_{I_z(\tau)}\theta^{-\frac12}\left|\tilde{\theta}_y\right
    | (s,\xi)d\xi+
    \int_{I_z(\tau)}\theta^{-\frac12}|\vartheta_y|(s,\xi)d\xi
    \\
    ~&\lesssim m_2^{-\frac{1}{2}}\delta+
    \left[\int_{I_z(\tau)}\frac{\vartheta_y^2}{v\theta^2}(s,\xi)d\xi\right]^{\frac{1}{2}}
    \left[\int_{I_z(\tau)}{v\theta}(s,\xi)d\xi\right]^{\frac{1}{2}}\\
    ~&\lesssim  m_2^{-\frac{1}{2}}\delta+
    \sup_{I_z(\tau)}v^{\frac12}(s,\cdot)
    \left[\int_{I_z(\tau)}\frac{\vartheta_y^2}{v\theta^2}(s,\xi)d\xi\right]^{\frac{1}{2}},
  \end{aligned}
  \end{equation}
  where we have used
  $b_{[z]+2}(s,\tau)\in \Omega_{[z]+2}(\tau)\subset I_z(\tau)$.
   Combine \eqref{pro14} with \eqref{bound2} and \eqref{lem1} to give
  \begin{equation}\label{pro2.2.31}
    \theta(s,y)\lesssim
    1+\sup_{I_z(\tau)}v(s,\cdot)
    \int_{I_z(\tau)}\frac{\vartheta_y^2}{v\theta^2}(s,\xi)d\xi
  \end{equation}
  and
  \begin{equation}\label{pro17}
    1-C
    \sup_{I_z(\tau)}v(s,\cdot)
    \int_{I_z(\tau)}\frac{\vartheta_y^2}{v\theta^2}(s,\xi)d\xi
    \lesssim \theta(s,y).
  \end{equation}
  We plug \eqref{pro2.2.31} into \eqref{pro2.2.30} to obtain
  \begin{equation*}
  \begin{aligned}
    v(t,y)\lesssim 1+\int_0^t\sup_{I_z(\tau)}v(s,\cdot)\int_{I_z(\tau)}
    \frac{\vartheta_y^2}{v\theta^2}(s,\xi)d\xi ds.
  \end{aligned}
  \end{equation*}
  Taking the supremum over $I_z(\tau)$ with respect to $y$, we have
 \begin{equation} \label{pro13}
  \sup_{I_z(\tau)}v(t,\cdot)\lesssim 1+\int_0^t
  \sup_{I_z(\tau)}v(s,\cdot)\int_{\Omega_i(\tau)}
    \frac{\vartheta_y^2}{v\theta^2}(s,\xi)d\xi ds.
 \end{equation}
 Applying Gronwall's inequality to \eqref{pro13},
 we can deduce from \eqref{est_El}
 that
 \begin{equation}\label{pro16}
  \sup_{I_z(\tau)} v (t,\cdot)\leq C_1\quad {\rm for\ all}\ t\in[0,\tau],
 \end{equation}
 where
 $C_1>0$ is some constant independent of $t$, $\tau$ and $z$.
 Noting that
 $z\in I_z(\tau)$, we deduce from \eqref{pro16} that
 $v(\tau,z)\leq C_1$. Since
 $(\tau,z)\in \Omega_T$ is arbitrary, we conclude
 \begin{equation}\label{upper1}
      v (\tau,z)\leq C_1 \quad {\rm for\ all}\ (\tau,z)\in\Omega_T.
  \end{equation}

\noindent{\bf Step 3}.
On the other hand, in view of \eqref{bound1}, \eqref{est_B} and \eqref{pro7},
we integrate \eqref{v_form} on $I_z(\tau)$
with respect to $y$ to find
  \begin{equation*}
    \begin{aligned}
      1\lesssim\int_{I_z(\tau)}v(t,y)dy\lesssim \mathrm{e}^{-t/C}
      +\int_0^t\frac{A_z(t)}{A_z(s)}ds.
    \end{aligned}
  \end{equation*}
Consequently, we have
\begin{equation}\label{Y-int}
  \int_0^t\frac{A_z(t)}{A_z(s)}ds\gtrsim 1-C\mathrm{e}^{-{t}/{C }}.
\end{equation}
Inserting \eqref{pro17}, \eqref{upper1} and
\eqref{Y-int} into \eqref{pro2.2.30}, we have
  \begin{equation}\label{pro9}
  \begin{aligned}
    v(t,y)
    \gtrsim~& \int_{0}^t\frac{A_z(t)}{A_z(s)}ds
    -C \int_{0}^t\frac{A_z(t)}{A_z(s)}\int_{I_z(\tau)}
\frac{\vartheta_y^2}{v\theta^2}d\xi ds\\
    \gtrsim~& 1-C \mathrm{e}^{-{t}/{C }}-
    C \left(\int_{0}^{\frac{t}2}+\int_{\frac{t}2}^t\right)
    \frac{A_z(t)}{A_z(s)}\int_{I_z(\tau)}
    \frac{\vartheta_y^2}{v\theta^2}d\xi ds
    \\
    \gtrsim~& 1-C \mathrm{e}^{-t/C }
    -C \int_{0}^{\frac{t}2}\mathrm{e}^{-\frac{t-s}{C}}
    \int_{I_z(\tau)}\frac{\vartheta_y^2}{v\theta^2}d\xi ds
    -C \int_{\frac{t}2}^t\int_{I_z(\tau)}\frac{\vartheta_y^2}{v\theta^2}d\xi  ds\\
    \gtrsim~& 1-C \mathrm{e}^{-t/C }
    -C \mathrm{e}^{-\frac{t}{2C}}
    -C \int_{\frac{t}2}^t\int_{I_z(\tau)}\frac{\vartheta_y^2}{v\theta^2}d\xi ds\\
    \gtrsim~& 1  \quad {\rm for\ all}\ (t,y)\in [T_0,\tau]\times I_z(\tau),
  \end{aligned}
\end{equation}
where $T_0$ is a positive constant independent of $t$.
In particular, the estimate \eqref{pro9} implies
\begin{equation}\label{pro15}
  v(\tau,z)\gtrsim 1 \quad {\rm for\ all}\ \tau\geq T_0,\ z>Y(\tau).
\end{equation}
As in \cite{KaMR651877,KS77MR0468593}, we can derive a positive lower bound for $v$,
that is,
  \begin{equation}\label{v_low1}
    v(\tau,z)\gtrsim e^{-Ct}, \quad {\rm for}\ (\tau,z) \in \Omega_T.
  \end{equation}
Finally, we combine \eqref{v_low1},
\eqref{upper1} and \eqref{pro15} to get \eqref{boun_v}.
This completes the proof.
\end{proof}

As a corollary of Lemma \ref{lem_boun}, we get the uniform bounds for the density
$\rho$ in the Eulerian coordinate.
\begin{corollary}
  If \eqref{lem1} holds for a sufficiently small $\epsilon_0>0$,
   then
  \begin{equation}\label{boun_rho}
   C_1^{-1}\leq \rho(t,x)\leq C_1\quad {\rm for\ all}\ (t,x)\in[0,T]\times\mathbb{R}_+,
  \end{equation}
   where the positive constant $C_1$ depends solely
on $\inf_{\mathbb{R}_+}(\rho_0,\theta_0)$
  and $\|(\phi_0,\psi_0,\vartheta_0)\|_1$.
\end{corollary}

\subsection{Pointwise estimates of temperature}
In the following lemma, we employ the maximum principle
to get the lower bound for the temperature, which does depend on the time $t$.
\begin{lemma}
  \label{lem_lower2}
  Suppose that \eqref{lem1} holds for a suitably small $\epsilon_0>0$.
Then
  \begin{equation} \label{lower2}
    \inf_{\mathbb{R}_+}\theta(t,\cdot)
    \geq\frac{\inf_{\mathbb{R}_+}\theta(s,\cdot)}{C_2\inf_{\mathbb{R}_+}
    \theta(s,\cdot)(t-s)+1}
    \quad
    {\rm for}\ 0\leq s\leq t\leq T,
  \end{equation}
  where the positive constant $C_2$ depends only
upon $\inf_{\mathbb{R}_+}(\rho_0,\theta_0)$
  and $\|(\phi_0,\psi_0,\vartheta_0)\|_1$.
\end{lemma}
\begin{proof}
  It follows from $\eqref{NS_E}_3$ that $\theta$ satisfies
  \begin{equation*}
    \theta_t+u\theta_x-\frac{\kappa}{c_v\rho}\theta_{xx}
    =\frac{\mu}{c_v\rho}\left[u_x^2-\frac{P}{\mu}u_x\right]
    \geq-\frac{P^2}{4\mu c_v\rho}=-\frac{R^2\rho}{4\mu c_v}\theta^2.
  \end{equation*}
  Hence we deduce from \eqref{boun_rho} that
  \begin{equation*}
    \theta_t+u\theta_x-\frac{\kappa}{c_v\rho}\theta_{xx}
    +C_2\theta^2\geq 0.
  \end{equation*}
  Let $\Theta:=\theta-\underline{\theta}$ with
  $\underline{\theta}:=\frac{\inf_{\mathbb{R}_+}\theta(s,\cdot)}
  {C_2\inf_{\mathbb{R}_+}\theta(s,\cdot)(t-s)+1}$. We observe
  \begin{equation*}
    \Theta|_{x=0,\infty}\geq 0,\quad
    \Theta|_{t=s}\geq 0,
  \end{equation*}
  and
  \begin{equation*}
    \begin{aligned}
      \Theta_t+u\Theta_x-\frac{\kappa}{c_v\rho}\Theta_{xx}+C_2(\theta+\underline{\theta})\Theta
      =\theta_t+u\theta_x
      -\frac{\kappa}{c_v\rho}\theta_{xx}+C_2\theta^2
       \geq 0.
    \end{aligned}
  \end{equation*}
  Applying the weak maximum principle for the parabolic equation,
we have that $\Theta(t,x)\geq 0$
  for $0\leq s\leq t\leq T$ and $x\in\mathbb{R}_+$.
This completes the proof of the lemma.
\end{proof}
Next we have the $L^2$-norm in both time and space of $\vartheta_x$.
\begin{lemma}
  \label{lem_the1}
  If \eqref{lem1} holds for a sufficiently small positive constant
  $\epsilon_0$, then
  \begin{equation}\label{lem2.3.1}
    \sup_{0\leq t\leq T}\int_{\mathbb{R}_+}\left[\phi^2+\vartheta^2+\psi^2\right]
    +\int_0^T\int_{\mathbb{R}_+}\left[(1+\theta+\psi^2)\psi_x^2+\vartheta_x^2\right]
    \lesssim 1.
  \end{equation}
\end{lemma}
\begin{proof}
  We divide the proof into five steps.

  \noindent{\bf Step 1}. First, for each $t\geq 0$ and $a>0$, we denote
  $$
  \Omega'_a(t):=\{x\in\mathbb{R}_+: \vartheta(t,x)>a\}.
  $$
  Then it follows from \eqref{lem3} and \eqref{boun_rho} that
  \begin{equation}\label{pro2.3.1}
  \begin{aligned}
    &\sup_{0\leq t\leq T}\left[\int_{\mathbb{R}_+}\phi^2
    +\int_{\mathbb{R}_+\backslash\Omega'_a(t)}\vartheta^2
    +\int_0^t|\phi(s,0)|^2ds\right]
    +\int_0^T\int_{\mathbb{R}_+\backslash\Omega'_a(t)}
    \left[\psi_x^2+\vartheta_x^2\right]\\[0.5mm]
    &\leq C(a)\sup_{0\leq t\leq T}\left[\int_{\mathbb{R}_+}\rho\mathcal{E}
    +\int_0^t\rho\Phi\left(\frac{\tilde{\rho}}{\rho}\right)(s,0)ds\right]
    +C(a)\int_0^T\int_{\mathbb{R}_+}\left[\frac{\psi_x^2}{\theta}
    +\frac{\vartheta_x^2}{\theta^2}\right]\leq C(a).
    \end{aligned}
  \end{equation}

  \noindent{\bf Step 2}. We now estimate the integral $\int_0^T\int_{\Omega'_a(t)}\vartheta_x^2$.
  For this purpose, we multiply $\eqref{per}_3$
  by $(\vartheta-2)_+:=\max\{\vartheta-2,0\}$ and integrate the resulting identity
  over $(0,t)\times \mathbb{R}_+$ to obtain
  \begin{equation}\label{pro2.3.2}
    \begin{aligned}
    &\frac{c_v}{2}\int_{\mathbb{R}_+}\rho(\vartheta-2)_+^2
    +\kappa\int_0^t\int_{\Omega'_2(s)}\vartheta_x^2
    +R\int_0^t\int_{\mathbb{R}_+}\rho\theta\psi_x(\vartheta-2)_+\\[0.5mm]
    =~&\frac{c_v}{2}\int_{\mathbb{R}_+}\rho_0(\vartheta_0-2)_+^2
    +\int_0^t\int_{\mathbb{R}_+}h(\vartheta-2)_+
    +\mu\int_0^t\int_{\mathbb{R}_+}\psi_x^2(\vartheta-2)_+.
  \end{aligned}
  \end{equation}
  To estimate the last term of \eqref{pro2.3.2}, we multiply $\eqref{per}_2$
  by $2\psi(\vartheta-2)_+$ and integrate the resulting identity
  over $(0,t)\times\mathbb{R}_+$ to find
  \begin{equation}\label{pro2.3.3}
    \begin{aligned}
      \int_{\mathbb{R}_+}&\psi^2\rho(\vartheta-2)_+
      +2\mu\int_0^t\int_{\mathbb{R}_+}\psi_x^2(\vartheta-2)_+
      -\int_0^t\int_{\Omega'_2(s)}\rho\psi^2(\vartheta_t+u\vartheta_x)\\[0.5mm]
       =~&\int_{\mathbb{R}_+}\psi_0^2\rho_0(\vartheta_0-2)_+
      +2R\int_0^t\int_{\mathbb{R}_+}\rho\theta\psi_x(\vartheta-2)_+
      +2R\int_0^t\int_{\Omega'_2(s)}\rho\theta\psi\vartheta_x\\[0.5mm]
      ~&+2R\int_0^t\int_{\mathbb{R}_+}\psi(\tilde{\rho}\tilde{\theta})_x(\vartheta-2)_+
      -2\mu\int_0^t\int_{\Omega'_2(s)}\psi\psi_x\vartheta_x
      +2\int_0^t\int_{\mathbb{R}_+}g\psi(\vartheta-2)_+.
    \end{aligned}
  \end{equation}
 Combining \eqref{pro2.3.3} and \eqref{pro2.3.2}, we have from $\eqref{per}_3$
 that
 \begin{equation}\label{pro2.3.4}
    \begin{aligned}
      &\int_{\mathbb{R}_+}\left[\frac{c_v}{2}\rho(\vartheta-2)_+^2
      +\psi^2\rho(\vartheta-2)_+\right]
      +\kappa\int_0^t\int_{\Omega'_2(s)}\vartheta_x^2
      +\mu\int_0^t\int_{\mathbb{R}_+}\psi_x^2(\vartheta-2)_+\\
      =&\int_{\mathbb{R}_+}\left[\frac{c_v}{2}\rho_0(\vartheta_0-2)_+^2
      +\psi_0^2\rho_0(\vartheta_0-2)_+\right]
      +\sum_{p=1}^6 \mathcal{J}_p,
    \end{aligned}
  \end{equation}
where each term $\mathcal{J}_p$ in the decomposition will be defined below.
We now define and estimate all the terms in the decomposition.
We first consider
\begin{equation*}
  \mathcal{J}_1:=
  R\int_0^t\int_{\mathbb{R}_+}\rho\theta\psi_x(\vartheta-2)_+
\end{equation*}
and
\begin{equation*}
  \mathcal{J}_2:=
  2R\int_0^t\int_{\Omega'_2(s)}\rho\theta\psi\vartheta_x.
\end{equation*}
  In light of \eqref{boun_rho}
  and \eqref{lem3}, we have
  \begin{equation}\label{pro2.3.5}
    \int_{\mathbb{R}_+}\psi^2+\int_{\Omega'_1(s)}\theta
    \lesssim \int_{\mathbb{R}_+}\rho\mathcal{E}\lesssim 1.
  \end{equation}
  From Cauchy's inequality and \eqref{boun_rho}, we obtain
  \begin{equation}\label{I1}
    \begin{aligned}
   |\mathcal{J}_1|\leq~& \epsilon \int_0^t\int_{\mathbb{R}_+}\psi_x^2(\vartheta-2)_+
   +C(\epsilon)\int_0^t\int_{\mathbb{R}_+}\theta^2(\vartheta-2)_+\\
   \leq~&\epsilon \int_0^t\int_{\mathbb{R}_+}\psi_x^2(\vartheta-2)_+
   +C(\epsilon)\int_0^t\int_{\mathbb{R}_+}\theta(\vartheta-1)_+^2\\
    \leq~&\epsilon \int_0^t\int_{\mathbb{R}_+}\psi_x^2(\vartheta-2)_+
  +C(\epsilon)\int_0^t\sup_{\mathbb{R}_+}(\vartheta-1)_+^2,
  \end{aligned}
  \end{equation}
  and
  \begin{equation}\label{I2}
    \begin{aligned}
     |\mathcal{J}_2|\leq~&\epsilon\int_0^t\int_{\Omega'_2(s)}\vartheta_x^2
     +C(\epsilon)\int_0^t\int_{\Omega'_2(s)}\psi^2\theta^2\\
     \leq~&\epsilon\int_0^t\int_{\Omega'_2(s)}\vartheta_x^2
     +C(\epsilon)\int_0^t\int_{\Omega'_2(s)}\psi^2\left(\vartheta-1\right)_+^2\\
     \leq~&\epsilon\int_0^t\int_{\Omega'_2(s)}\vartheta_x^2
     +C(\epsilon)\int_0^t\sup_{\mathbb{R}_+}(\vartheta-1)_+^2.\\
   \end{aligned}
  \end{equation}
  Let us define
  \begin{equation*}
    \mathcal{J}_3:=
    \int_0^t\int_{\Omega'_2(s)}\left[h(\vartheta-2)_+
      +c_v^{-1}h\psi^2+2g\psi(\vartheta-2)_+\right].
  \end{equation*}
  It follows from \eqref{boun_rho} and the identity $P-\tilde{P}=R\rho\vartheta
  +R\tilde{\theta}\phi$ that
  \begin{equation}\label{g_h}
    |g|\lesssim |\tilde{u}_x|(\phi,\psi)|,\ \
    |h|\lesssim |(\tilde{u}_x,\tilde{\theta}_x)||(\phi,\psi,\vartheta,\psi_x)|.
  \end{equation}
  Applying Cauchy's inequality to $\mathcal{J}_3$ yields
  \begin{equation}\label{I3'}
    \begin{aligned}
      |\mathcal{J}_3|\lesssim&\int_0^t\int_{\Omega'_2(s)}
      |(\tilde{u}_x,\tilde{\theta}_x)||(\phi,\psi,\vartheta,\psi_x)|
      \left[(\vartheta-2)_+(1+|\psi|)+\psi^2\right]\\
      \lesssim&\int_0^t\int_{\Omega'_2(s)}|(\tilde{u}_x,\tilde{\theta}_x)|
      \left[(\phi^2+\psi^2+\vartheta^2)+\psi_x^2+(\vartheta-2)_+^2(1+\psi^2)+\psi^4\right].
    \end{aligned}
  \end{equation}
  We obtain from \eqref{decay1} and \eqref{pro2.3.5} that
   \begin{align}
\label{I3_2}
    &\int_0^t\int_{\mathbb{R}_+}|(\tilde{u}_x,\tilde{\theta}_x)|(\vartheta-2)_+^2(1+\psi^2)
      \lesssim \int_0^t\sup_{\mathbb{R}_+}(\vartheta-1)_+^2,
  \\ \label{I3_3}
   & \int_0^t\int_{\mathbb{R}_+}|(\tilde{u}_x,\tilde{\theta}_x)|\psi^4
      \lesssim \int_0^t\|\psi\|_{L^\infty}^4\|(\tilde{u}_x,\tilde{\theta}_x)\|_{L^1}
      \lesssim \delta \int_0^t\|\psi\|^2\|\psi_x\|^2.
  \end{align}
  Plugging \eqref{pro2.1.4}, \eqref{I3_2}-\eqref{I3_3} into \eqref{I3'},
   we deduce from \eqref{pro2.3.1},
  \eqref{pro2.3.5} and Lemma \ref{lem_bas} that
  \begin{equation}\label{I3}
    \begin{aligned}
      |\mathcal{J}_3|\lesssim~&\delta\int_0^t
      \left[|\phi(s,0)|^2+\|(\phi_x,\psi_x,\vartheta_x)(s)\|^2\right]ds
      +\int_0^t\sup_{\mathbb{R}_+}(\vartheta-1)_+^2\\
      \lesssim~&1+\delta\int_0^t\int_{\mathbb{R}_+}
      \left[m_2^{-1}\frac{\theta\phi_x^2}{\rho^2}
      +M_2^2\left(\frac{\psi_x^2}{\theta}+\frac{\vartheta_x^2}{\theta^2}\right)\right]
      +\int_0^t\sup_{\mathbb{R}_+}(\vartheta-1)_+^2\\
  \lesssim~& 1+\int_0^t\sup_{\mathbb{R}_+}(\vartheta-1)_+^2.
    \end{aligned}
  \end{equation}
  Let us now consider the term
  \begin{equation*}
    \mathcal{J}_4:=
    2R\int_0^t\int_{\mathbb{R}_+}\psi(\tilde{\rho}\tilde{\theta})_x(\vartheta-2)_+,
  \end{equation*}
  which is trivially estimated by
  \begin{equation}\label{I4}
    |\mathcal{J}_4|\lesssim\int_0^t\sup_{\mathbb{R}_+}(\vartheta-2)_+
    \|\psi\|\|(\tilde{\rho}_x,\tilde{\theta}_x)\|
    \lesssim\int_0^t\sup_{\mathbb{R}_+}(\vartheta-1)_+^2.
  \end{equation}
  For the term
  \begin{equation*}
    \mathcal{J}_5:=
    \int_0^t\int_{\Omega'_2(s)}\left[\mu\psi_x^2c_v^{-1}\psi^2
     -P\psi_xc_v^{-1}\psi^2-2\mu\psi\psi_x\vartheta_x\right],
  \end{equation*}
  we apply Cauchy's inequality and \eqref{pro2.3.5} to deduce
  \begin{equation}\label{I5}
  \begin{aligned}
    |\mathcal{J}_5|\lesssim~& \epsilon \int_0^t\int_{\Omega'_2(s)}\vartheta_x^2
   +C(\epsilon)\int_0^t\int_{\Omega'_2(s)}\psi^2\psi_x^2
   +\int_0^t\int_{\Omega'_2(s)}\psi^2\theta^2\\
    \lesssim~&\epsilon \int_0^t\int_{\Omega'_2(s)}\vartheta_x^2
   +C(\epsilon)\int_0^t\int_{\Omega'_2(s)}\psi^2\psi_x^2
   +\int_0^t\int_{\Omega'_2(s)}\psi^2(\vartheta-1)_+^2\\
   \lesssim~&\epsilon \int_0^t\int_{\Omega'_2(s)}\vartheta_x^2
   +C(\epsilon)\int_0^t\int_{\Omega'_2(s)}\psi^2\psi_x^2
   +\int_0^t\sup_{\mathbb{R}_+}(\vartheta-1)_+^2.
  \end{aligned}
  \end{equation}
  We finally consider
  \begin{equation*}
    \mathcal{J}_6:=
    \int_0^t\int_{\Omega'_2(s)}c_v^{-1}\kappa\psi^2\vartheta_{xx}.
  \end{equation*}
  In order to estimate $\mathcal{J}_6$,
  we apply Lebesgue's dominated convergence theorem to find
  \begin{equation}\label{I6}
    \begin{aligned}
    \mathcal{J}_6
    =~&\frac{\kappa}{c_v}\lim_{\nu\to 0^+}\int_0^t\int_{\mathbb{R}_+}
    \varphi_{\nu}(\vartheta)\psi^2\vartheta_{xx}\\
    =~&\frac{\kappa}{c_v}\lim_{\nu\to 0^+}\int_0^t\int_{\mathbb{R}_+}
    \left[-2\varphi_{\nu}(\vartheta)\psi\psi_x\vartheta_{x}
    -\varphi'_{\nu}(\vartheta)\psi^2\vartheta_{x}^2\right]\\
    \leq ~&-\frac{\kappa}{c_v}\lim_{\nu\to 0^+}\int_0^t\int_{\mathbb{R}_+}
    2\varphi_{\nu}(\vartheta)\psi\psi_x\vartheta_{x}\\
    \leq ~&\epsilon \int_0^t\int_{\mathbb{R}_+}\vartheta_x^2
    +C(\epsilon)\int_0^T\int_{\mathbb{R}_+}\psi^2\psi_x^2.
  \end{aligned}
  \end{equation}
  where the approximate scheme $\varphi_{\nu}(\vartheta)$ is defined by
  \begin{equation*}
  \varphi_{\nu}(\vartheta)=
  \begin{cases}
    1, \quad \quad &\vartheta-2>\nu,\\
    (\vartheta-2)/\nu, \quad  \quad& 0<\vartheta-2\leq\nu,\\
    0, \quad \quad&\vartheta-2\leq 0.\\
  \end{cases}
\end{equation*}
Plugging \eqref{I1}-\eqref{I2}, \eqref{I3}-\eqref{I6} into \eqref{pro2.3.4}, we get
from \eqref{boun_rho} that
\begin{equation}\label{pro2.3.6}
  \begin{aligned}
    &\int_{\mathbb{R}_+}(\vartheta-2)_+^2
      +\int_0^t\int_{\Omega'_2(s)}\left[\vartheta_x^2
      +\psi_x^2(\vartheta-2)_+\right]\\
      \lesssim~& 1+\epsilon\int_0^t\int_{\mathbb{R}_+}\vartheta_x^2
      +C(\epsilon)\int_0^t\int_{\mathbb{R}_+}\psi^2\psi_x^2
      +C(\epsilon)\int_0^t\sup_{\mathbb{R}_+}(\vartheta-1)_+^2.
  \end{aligned}
\end{equation}

\noindent{\bf Step 3}.
We obtain from \eqref{lem3} that
\begin{equation}\label{pro2.3.7}
  \begin{aligned}
 \int_0^t\int_{\mathbb{R}_+}\left[\vartheta_x^2+\psi_x^2\theta\right]
 \lesssim~&\int_0^t\int_{\Omega'_3(s)}
 \left[\vartheta_x^2+\psi_x^2(\vartheta-2)_+\right]
 +\int_0^t\int_{\mathbb{R}_+\setminus\Omega'_3(s)}
 \left[\frac{\vartheta_x^2}{\theta^2}+\frac{\psi_x^2}{\theta}\right]\\
 \lesssim~&\int_0^t\int_{\Omega'_2(s)}
 \left[\vartheta_x^2+\psi_x^2(\vartheta-2)_+\right]
 +1.
\end{aligned}
\end{equation}
Combining \eqref{pro2.3.7} and \eqref{pro2.3.6}, and choosing $\epsilon$
sufficiently small, we have
\begin{equation}\label{pro2.3.8}
      \int_{\mathbb{R}_+}(\vartheta-2)_+^2
      +\int_0^t\int_{\mathbb{R}_+}\left[\vartheta_x^2
      +\psi_x^2\theta\right]
      \lesssim 1
      +\int_0^t\sup_{\mathbb{R}_+}(\vartheta-1)_+^2
      +\int_0^t\int_{\mathbb{R}_+}\psi^2\psi_x^2.
\end{equation}

\noindent{\bf Step 4}.
To estimate the last term of \eqref{pro2.3.8}, we
multiply $\eqref{per}_2$ by $\psi^3$ and then integrate the resulting identity
over $(0,t)\times\mathbb{R}_+$ to have
\begin{equation}\label{pro2.3.9}
 \begin{aligned}
   &\frac{1}{4}\int_{\mathbb{R}_+}\rho\psi^4
   +3\mu\int_0^t\int_{\mathbb{R}_+}\psi^2\psi_x^2
   -\frac{1}{4}\int_{\mathbb{R}_+}\rho_0\psi_0^4\\
   =~&
   3R\int_0^t\int_{\mathbb{R}_+}\psi^2\psi_x\tilde{\theta}\phi
   +3R\int_0^t\int_{\mathbb{R}_+}\psi^2\psi_x\rho\vartheta
   +\int_0^t\int_{\mathbb{R}_+}g\psi^3.
 \end{aligned}
\end{equation}
From \eqref{pro2.3.1} and \eqref{pro2.3.5}, we have
\begin{equation}\label{pro2.3.10}
 \begin{aligned}
    &\int_0^t\int_{\mathbb{R}_+}\psi^2\psi_x\tilde{\theta}\phi
+\int_0^t\int_{\mathbb{R}_+\setminus\Omega'_2(s)}\psi^2\psi_x\rho\vartheta\\
    &\lesssim\int_0^t\|\psi\|_{L_x^\infty}^2\|\psi_x\|
     \left[\|\phi\|+\|\vartheta\|_{L^2(\mathbb{R}_+\setminus\Omega_2'(s))}\right]
    \lesssim\int_0^t\|\psi_x\|^2,
  \end{aligned}
\end{equation}
We then apply Cauchy's inequality to derive
\begin{equation}\label{pro2.3.11}
  \begin{aligned}
  \int_0^t\int_{\Omega_2'(s)}\psi^2\psi_x\rho\vartheta
  \leq~&
  \epsilon\int_0^t\int_{\mathbb{R}_+}\psi^2\psi_x^2
  +C(\epsilon)\int_0^t\int_{\Omega_2'(s)}\psi^2\vartheta^2\\
  \leq~& \epsilon\int_0^t\int_{\mathbb{R}_+}\psi^2\psi_x^2
  +C(\epsilon)\int_0^t\sup_{\mathbb{R}_+}\left(\vartheta-1\right)_+^2.
\end{aligned}
\end{equation}
In view of  \eqref{pro2.3.1}, \eqref{pro2.3.5}, \eqref{g_h}
and \eqref{lem2}, we have
\begin{equation}\label{pro2.3.12}
  \begin{aligned}
  \int_0^t\int_{\mathbb{R}_+}g\psi^3
    \lesssim~&\int_0^t\int_{\mathbb{R}_+}|\tilde{u}_x|\left(\psi^4+\phi^4\right)\\
    \lesssim~&\delta\int_0^t\left(\|\psi\|^2\|\psi_x\|^2+\|\phi\|^2\|\phi_x\|^2\right)\\
    \lesssim~&\delta\int_0^t\|\psi_x\|^2+
    \delta M_1^2m_2^{-1}\int_0^t\int_{\mathbb{R}_+}\frac{\theta\phi_x^2}{\rho^2}\\
    \lesssim~&\delta\int_0^t\|\psi_x\|^2+\delta M_1^2m_2^{-1}M_2.
\end{aligned}
\end{equation}
Plugging \eqref{pro2.3.10} -\eqref{pro2.3.12} into \eqref{pro2.3.9}, and taking $\epsilon$
sufficiently small, we derive from \eqref{lem1} that
\begin{equation}\label{pro2.3.13}
  \begin{aligned}
   \int_{\mathbb{R}_+}\psi^4+\int_0^t\int_{\mathbb{R}_+}\psi^2\psi_x^2
   \lesssim1+\int_0^t\int_{\mathbb{R}_+}\psi_x^2
   +\int_0^t\sup_{\mathbb{R}_+}\left(\vartheta-1\right)_+^2.
 \end{aligned}
\end{equation}
We note from \eqref{lem3} that
\begin{equation}\label{pro2.3.14}
  \int_0^t\int_{\mathbb{R}_+}\psi_x^2
 \leq \epsilon \int_0^t\int_{\mathbb{R}_+}\theta\psi_x^2
 +C(\epsilon)\int_0^t\int_{\mathbb{R}_+}\frac{\psi_x^2}{\theta}
   \leq \epsilon \int_0^t\int_{\mathbb{R}_+}\theta\psi_x^2
   +C(\epsilon).
\end{equation}
Combination of  \eqref{pro2.3.14} and \eqref{pro2.3.13} yields
\begin{equation}\label{pro2.3.15}
  \int_{\mathbb{R}_+}\psi^4+\int_0^t\int_{\mathbb{R}_+}(1+\psi^2)\psi_x^2
  \lesssim C(\epsilon)+\epsilon \int_0^t\int_{\mathbb{R}_+}\theta\psi_x^2
  +\int_0^t\sup_{\mathbb{R}_+}\left(\vartheta-1\right)_+^2.
\end{equation}
We plug \eqref{pro2.3.15} into \eqref{pro2.3.8} and choose $\epsilon$
suitable small to find
\begin{equation}\label{pro2.3.16}
    \int_{\mathbb{R}_+}
    \left[\left(\vartheta-2\right)_+^2+\psi^4\right]
    +\int_0^t\int_{\mathbb{R}_+}\left[\vartheta_x^2+\psi_x^2\left(1+\theta+\psi^2\right)\right]
    \lesssim 1+\int_0^t\sup_{\mathbb{R}_+}\left(\vartheta-1\right)_+^2.
 \end{equation}

\noindent{\bf Step 5}.
It remains to estimate the last term of \eqref{pro2.3.16}.
According to the fundamental theorem of calculus, we have from \eqref{pro2.3.5} that
\begin{equation}\label{pro2.3.17}
  \begin{aligned}
    \int_0^T\sup_{\mathbb{R}_+}\left(\vartheta-1\right)_+^2
    \leq ~&\int_0^T\left[\int_{\Omega_{1}'(s)}\left|\vartheta_x\right|\right]^2\\
    \leq ~&\int_0^T\left[\int_{\Omega_{1}'(s)}\frac{\vartheta_x^2}{\theta}
    \int_{\Omega_{1}'(s)}\theta\right]\\
    \leq ~&\epsilon\int_0^T\int_{\mathbb{R}_+}\vartheta_x^2
    +C(\epsilon)\int_0^T\int_{\mathbb{R}_+}\frac{\vartheta_x^2}{\theta^2}\\
    \leq ~&\epsilon\int_0^T\int_{\mathbb{R}_+}\vartheta_x^2+C(\epsilon).
  \end{aligned}
\end{equation}
Plug \eqref{pro2.3.17} into \eqref{pro2.3.16} and choose $\epsilon>0$
suitable small to obtain \eqref{lem2.3.1}.
This completes the proof of the lemma.
\end{proof}
We obtain the upper bound for the temperature uniformly in both time and
space in the next lemma,
by combining Lemma \ref{lem_the1} and some desired uniform estimates
on the spatial derivatives of $(\phi,\psi,\vartheta)$.
\begin{lemma}\label{lem_the2}
If \eqref{lem1} holds for a sufficiently small $\epsilon>0$, then we have
  \begin{gather}\label{lem2.3.3}
   \theta(t,x)\leq C_3 \quad {\rm for\ all}\ (t,x)\in[0,T]\times\mathbb{R}_+,\\
   \label{lem2.3.2}
    \sup_{0\leq t\leq T}\|(\phi,\psi,\vartheta)(t)\|_1^2
    +\int_0^T
    \left[\|\sqrt\theta\phi_x(t)\|^2+\|(\psi_{x},\vartheta_{x})(t)\|_1^2\right]dt
    \leq C_4^2.
  \end{gather}
\end{lemma}
\begin{proof}
  First, plugging \eqref{boun_rho} into \eqref{pro2.1.8},
  we deduce
  \begin{equation}\label{pro2.4.1}
    \int_{\mathbb{R}_+}\phi_x^2+\int_0^t\int_{\mathbb{R}_+}\theta\phi_x^2
    \lesssim 1+\int_0^t\int_{\mathbb{R}_+}\left[\psi_x^2+\vartheta_x^2
    +\frac{\vartheta_x^2}{\theta^2}\right]\lesssim 1,
  \end{equation}
where we employed \eqref{lem3} and \eqref{lem2.3.1} in the last inequality.

  Next, multiply $\eqref{per}_2$ by $\frac{\psi_{xx}}{\rho}$ to derive
  \begin{equation*}
    \left(\frac{\psi_x^2}{2}\right)_t-\left[\psi_t\psi_x+\frac12u\psi_x^2\right]_x
    +\frac12u_x\psi_x^2+\mu\frac{\psi_{xx}^2}{\rho}
    =\frac{(P-\tilde{P})_x}{\rho}\psi_{xx}-\frac{g}{\rho}\psi_{xx}.
  \end{equation*}
  Integrating this last identity over $(0,T)\times\mathbb{R}_+$, we obtain from \eqref{boun_rho}
  and Cauchy's inequality that
  \begin{equation}\label{pro2.4.2}
  \begin{aligned}
    &\int_{\mathbb{R}_+}\psi_x^2
    -\int_0^T\psi_x^2(t,0)dt+\int_0^T\int_{\mathbb{R}_+}\psi_{xx}^2\\
    &\lesssim 1
    +\int_0^T\int_{\mathbb{R}_+}\left[(P-\tilde{P})_x^2
    +g^2+|\tilde{u}_x|\psi_x^2+|\psi_x|^3\right].
  \end{aligned}
  \end{equation}
  Apply Sobolev's inequality and \eqref{lem2.3.1} to obtain
  \begin{equation}\label{pro2.4.3}
   \begin{aligned}
     \int_0^T\psi_x^2(t,0)dt+\int_0^T\int_{\mathbb{R}_+}|\psi_x|^3
     \lesssim~&\int_0^T\|\psi_x\|\|\psi_{xx}\|
     +\int_0^T\|\psi_x\|^{\frac52}\|\psi_{xx}\|^{\frac12}\\
     \lesssim~&\epsilon\int_0^T\int_{\mathbb{R}_+}\psi_{xx}^2
     +C(\epsilon)\int_0^T\left[\|\psi_x\|^2
     +\|\psi_x\|^{\frac{10}{3}}\right]\\
     \lesssim ~&C(\epsilon)\left[1+
     \sup_{0\leq t\leq T}\|\psi_x\|^{\frac{4}{3}}\right]
     +\epsilon\int_0^T\int_{\mathbb{R}_+}\psi_{xx}^2.
   \end{aligned}
  \end{equation}
  In view of the identity
  $
    (P-\tilde{P})_x=R(\theta\phi_x+\rho\vartheta_x
    +\phi\tilde{\theta}_x+\vartheta\tilde{\rho}_x),
  $
  \eqref{g_h}, \eqref{pro2.4.1}, \eqref{lem2.3.1}
  and \eqref{pro2.1.4}, we derive
  \begin{equation}\label{pro2.4.4}
    \begin{aligned}
     &\int_0^T\int_{\mathbb{R}_+}\left[(P-\tilde{P})_x^2
    +g^2+|\tilde{u}_x|\psi_x^2\right]\\
    ~&\lesssim \int_0^T\int_{\mathbb{R}_+}\left[|(\theta\phi_x,\psi_x,\vartheta_x)|^2
    +|(\tilde{\rho}_x,\tilde{u}_x,\tilde{\theta}_x)|^2|(\phi,\psi,\vartheta)|^2
    \right]\\
    ~&\lesssim 1+\|\theta\|_{L^\infty([0,T]\times\mathbb{R}_+)}.
   \end{aligned}
  \end{equation}
  We plug \eqref{pro2.4.3}-\eqref{pro2.4.4} into \eqref{pro2.4.2} to get
  \begin{equation*}
    \sup_{0\leq t\leq T}\|\psi_x(t)\|^2+\int_0^T\int_{\mathbb{R}_+}\psi_{xx}^2
    \lesssim 1+\|\theta\|_{L^\infty([0,T]\times\mathbb{R}_+)}
    +\sup_{0\leq t\leq T}\|\psi_x\|^{\frac43}.
  \end{equation*}
  Then Young's inequality yields the estimate
  \begin{equation}\label{pro2.4.5}
    \sup_{0\leq t\leq T}\|\psi_x(t)\|^2+\int_0^T\int_{\mathbb{R}_+}\psi_{xx}^2
    \lesssim 1+\|\theta\|_{L^\infty([0,T]\times\mathbb{R}_+)}.
  \end{equation}

  Next, multiply $\eqref{per}_3$ by $\frac{\vartheta_{xx}}{\rho}$ and integrate
  the resulting identity over $(0,T)\times\mathbb{R}_+$ to have
  \begin{equation*}
    \begin{aligned}
      \frac{c_v}{2}\int_{\mathbb{R}_+}\vartheta_x^2
      +\kappa\int_0^T\int_{\mathbb{R}_+}\frac{\vartheta_{xx}^2}{\rho}
      =\frac{c_v}{2}\int_{\mathbb{R}_+}\vartheta_{0x}^2
      +\int_0^T\int_{\mathbb{R}_+}\left[c_v u\vartheta_x
      -\mu\frac{\psi_x^2}{\rho} + R\theta\psi_x
      -\frac{h}{\rho}\right]\vartheta_{xx},
    \end{aligned}
  \end{equation*}
  which combined with \eqref{boun_rho} implies
  \begin{equation}\label{pro2.4.6}
    \begin{aligned}
      \int_{\mathbb{R}_+}\vartheta_x^2
      +\int_0^T\int_{\mathbb{R}_+}\vartheta_{xx}^2
      \lesssim~&1+\int_0^T\int_{\mathbb{R}_+}
      \left[u^2\vartheta_x^2+\|\psi_x\|_{L^\infty}^2\psi_x^2+\theta^2\psi_x^2+h^2\right]\\
      \lesssim~&1+\int_0^T(1+\|\psi\|\|\psi_x\|)\|\vartheta_x\|^2
      +\int_0^T\|\psi_x\|^3\|\psi_{xx}\|\\
      &+\|\theta\|_{L^\infty([0,T]\times\mathbb{R}_+)}\int_0^T
      \int_{\mathbb{R}_+}\theta\psi_x^2
      +\int_0^T\int_{\mathbb{R}_+}h^2.
    \end{aligned}
  \end{equation}
  By \eqref{g_h}, \eqref{pro2.1.4} and \eqref{pro2.4.5}, we have
  $\int_0^T\int_{\mathbb{R}_+}h^2\lesssim 1,$ and
  \begin{align*}
    \int_0^T\|\psi_x\|^3\|\psi_{xx}\|
   \lesssim\sup_{0\leq t\leq T}\|\psi_x\|^2\int_0^T\left(\|\psi_x\|^2+\|\psi_{xx}\|^2\right)
   \lesssim 1+\|\theta\|_{L^\infty([0,T]\times\mathbb{R}_+)}^2.
  \end{align*}
 In light of \eqref{pro2.4.1}, \eqref{pro2.4.5} and \eqref{lem2.3.1},
we then obtain
 \begin{equation}\label{pro2.4.7}
   \int_{\mathbb{R}_+}\vartheta_x^2
      +\int_0^T\int_{\mathbb{R}_+}\vartheta_{xx}^2
   \lesssim 1+\|\theta\|_{L^\infty([0,T]\times\mathbb{R}_+)}^2.
 \end{equation}

 Finally, it follows from \eqref{lem2.3.1} and \eqref{pro2.4.7} that
  \begin{equation*}
  \begin{aligned}
   \|\theta-\tilde{\theta}\|_{L^\infty([0,T]\times\mathbb{R}_+)}^2
    \lesssim \sup_{0\leq t\leq T}\|\vartheta(t)\|\|\vartheta_x(t)\|
    \lesssim1+\|\theta\|_{L^\infty([0,T]\times\mathbb{R}_+)}.
  \end{aligned}
  \end{equation*}
  This implies \eqref{lem2.3.3} by virtue of Cauchy's inequality.
  Combine
  \eqref{pro2.4.2}, \eqref{pro2.4.5} and \eqref{pro2.4.7} to give
  \begin{equation*}
        \sup_{0\leq t\leq T}\int_{\mathbb{R}_+}[\phi_x^2+\psi_x^2+\vartheta_x^2]
    +\int_0^T\int_{\mathbb{R}_+}[\theta\phi_x^2+\psi_{xx}^2+\vartheta_{xx}^2]
    \lesssim 1,
  \end{equation*}
  which together with \eqref{lem2.3.1} yields \eqref{lem2.3.2}.
This completes the proof of the lemma.
\end{proof}

\section{Proof of Theorem \ref{thm}}
In this section, we prove Theorem \ref{thm} in six steps
 by employing the continuation argument.

\noindent{\bf Step 1}.
Set
$T_1=128\lambda_3^{-4}C_4^4$.
We choose
positive constants $\Pi$, $\lambda_i,$ and $\Lambda_i$  $(i=1,2,3)$
such that $\|(\phi_0,\psi_0,\vartheta_0)\|_1\leq \Pi$ and
\begin{equation*}\label{3.1}
  \lambda_1\leq \rho_0(x)\leq \Lambda_1,\quad
  \lambda_2\leq \theta_0(x)\leq \Lambda_2,\quad
  \lambda_3\leq \tilde{\theta}(x)\leq \Lambda_3\quad
  {\rm for\ all}\ x\in\mathbb{R}_+.
\end{equation*}
Applying Proposition \ref{Pro_loc}, we see that there exists a constant
$0<t_1\leq \min\{T_1,T_0(\lambda_1,\lambda_2,\Pi)\}$
such that the problem \eqref{per} has
a unique solution $(\phi,\psi,\vartheta)\in X(0,t_1;\frac{1}{2}\lambda_1,
2\Lambda_1;\frac{1}{2}\lambda_2,2\Lambda_2)$.

Let $0<\delta\leq \delta_1$ with $$\Xi\left(\frac{\lambda_1}{2},
2\Lambda_1,\frac{\lambda_2}{2},2\Lambda_2\right)\delta_1=\epsilon_0.$$
Then we can apply Lemmas
\ref{lem_boun}, \ref{lem_lower2} and \ref{lem_the2}
with $T=t_1$ to obtain that
the local solution $(\phi,\psi,\vartheta)$ constructed above satisfies
for each $t\in[0,t_1]$ that
\begin{equation}\label{3.2}
      \theta(t,x)\geq \frac{\lambda_2}
      {C_2 \lambda_2T_1+1} \quad {\rm for\ all}\ x\in \mathbb{R}_+,
\end{equation}
and
\begin{equation}\label{3.3}
  \begin{gathered}
    C_1^{-1}\leq \rho(t,x)\leq C_1,
    \quad
    \theta(t,x)\leq C_3\quad {\rm for\ all}\ x\in \mathbb{R}_+,\\
    \|(\phi,\psi,\vartheta)(t)\|_1^2
    +\int_0^{t}
    \left[\|\sqrt\theta\phi_x(s)\|^2+\|(\psi_{x},\vartheta_{x})(s)\|_1^2\right]ds
    \leq C_4^2.
  \end{gathered}
\end{equation}

\noindent{\bf Step 2}.
If we take $(\phi,\psi,\vartheta)(t_1,\cdot)$ as the initial data,
we can apply Proposition \ref{Pro_loc}
and extend the local solution  $(\phi,\psi,\vartheta)$ to the time
interval $[0,t_1+t_2]$ with $t_2\leq\min\{T_1-t_1, T_0(\frac{1}{C_1},
\frac{\lambda_2}{C_2 \lambda_2T_1+1},C_4)\}$.
Moreover, we have 
\begin{equation*}\label{3.4}
  \frac{1}{2C_1}\leq \rho(t,x)\leq 2C_1,\quad
  \frac{\lambda_2}
      {2(C_2 \lambda_2T_1+1)}\leq \theta(t,x)\leq 2C_3
\end{equation*}
for all $(t,x)\in[t_1,t_1+t_2]\times\mathbb{R}_+.$
Take $0<\delta\leq \min\{\delta_1,\delta_2\}$ with
$$\Xi\left(\frac{1}{2C_1},
2C_1,\frac{\lambda_2}
      {2(C_2 \lambda_2T_1+1)},2C_3\right)\delta_2=\epsilon_0.$$
Then we can employ
Lemmas
\ref{lem_boun}, \ref{lem_lower2} and \ref{lem_the2}
with $T=t_1+t_2$ that
the local solution $(\phi,\psi,\vartheta)$ satisfies \eqref{3.2} and
\eqref{3.3} for each $t\in[0,t_1+t_2]$.

\noindent{\bf Step 3}.
We repeat the argument in Step 2, to extend our
solution $(\phi,\psi,\vartheta)$ to the time interval $[0,t_1+t_2+t_3]$
with
$t_3\leq\min\{T_1-(t_1+t_2), T_0(\frac{1}{C_1},
\frac{\lambda_2}{C_2 \lambda_2T_1+1},C_4)\}$.
Assume that  $0<\delta\leq \min\{\delta_1,\delta_2\}$.
Continuing, after finitely many steps we construct the unique solution
$(\phi,\psi,\vartheta)$ existing on $[0,T_1]$ and satisfying
 \eqref{3.2} and \eqref{3.3} for each $t\in[0,T_1]$.

\noindent{\bf Step 4}.
Since $T_1\geq 128\lambda_3^{-4}C_4^4$ and
$$\sup_{0\leq t\leq T_1}\|\vartheta(t)\|_1^2
    +\int_{\frac{T_1}{2}}^{T_1}\|\vartheta_{x}(t)\|_1^2dt
    \leq C_4^2,$$
we can find a $t_0'\in[T_1/2,T_1]$ such that
$$\|\vartheta(t_0')\|\leq C_4,\quad
\|\vartheta_x(t_0')\|\leq \frac{\lambda_3^2}{8C_4}.$$
Sobolev's inequality yields
$$
\|\vartheta(t_0')\|_{L^{\infty}}\leq
\sqrt2\|\vartheta(t_0')\|^{\frac12}\|\vartheta_x(t_0')\|^{\frac12}
\leq \frac{\lambda_3}{2},$$
and so
\begin{equation*}\label{3.5}
\theta(t_0',x)\geq \tilde\theta(x)-\|\vartheta(t_0')\|_{L^{\infty}}
\geq \frac{\lambda_3}{2}\quad {\rm for\ all}\ x\in\mathbb{R}_+.
\end{equation*}
We note here that
\begin{equation*}\label{3.6}
\|(\phi,\psi,\vartheta)(t_0')\|_1
    \leq C_4,\quad
  C_1^{-1}\leq \rho(t_0',x)\leq C_1,
    \quad
    \theta(t_0',x)\leq C_3\quad {\rm for\ all}\ x\in \mathbb{R}_+.
\end{equation*}
Now we apply Proposition \ref{Pro_loc} again by
taking $(\phi,\psi,\vartheta)(t_0',\cdot)$ as the initial data.
Then we see that the solution  $(\phi,\psi,\vartheta)$ exists on
$[t_0',t_0'+t_1']$ with $t_1'\leq \min\{T_1,T_0(\frac{1}{C_1},
\frac{1}{2}\lambda_3,C_4)\}$ and satisfies
\begin{equation*}\label{3.7}
  \frac{1}{2C_1}\leq \rho(t,x)\leq 2C_1,\quad
  \frac{\lambda_3}{4}\leq \theta(t,x)\leq 2C_3
\end{equation*}
for all $(t,x)\in[t_0',t_0'+t_1']\times\mathbb{R}_+.$
If we take $0<\delta\leq \min\{\delta_1,\delta_2,\delta_3\}$ with
$$\Xi\left(\frac{1}{2C_1},
2C_1,\frac{\lambda_3}4,2C_3\right)\delta_3=\epsilon_0,$$
then we can deduce from
Lemmas
\ref{lem_boun}, \ref{lem_lower2} and \ref{lem_the2}
with $T=t_0'+t_1'$
 that for each time $t\in[t_0',t_0'+t_1']$,
the local solution $(\phi,\psi,\vartheta)$ satisfies
\eqref{3.3} and
\begin{equation}\label{3.8}
      \theta(t,x)\geq \frac{\inf_{\mathbb{R}_+}\theta(t_0',\cdot)}
      {C_2 \inf_{\mathbb{R}_+}\theta(t_0',\cdot)T_1+1}
      \geq \frac{\lambda_3}
      {C_2 \lambda_3T_1+2}
      \quad {\rm for\ all}\ x\in \mathbb{R}_+.
\end{equation}

\noindent{\bf Step 5}.
Next if we take $(\phi,\psi,\vartheta)(t_0'+t_1',\cdot)$ as the initial data,
we  apply Proposition \ref{Pro_loc}
 and construct the solution  $(\phi,\psi,\vartheta)$ existing on the time
interval $[0,t_0'+t_1'+t_2']$ with $t_2'\leq
\min\{T_1-t'_1,T_0(\frac{1}{C_1},
\frac{\lambda_3}
      {C_2 \lambda_3T_1+2},C_4)\}$ and satisfying
\begin{equation*}
  \frac{1}{2C_1}\leq \rho(t,x)\leq 2C_1,\quad
\frac{\lambda_3}
      {2(C_2 \lambda_3T_1+2)}\leq \theta(t,x)\leq 2C_3
\end{equation*}
 for all $(t,x)\in[t_0'+t_1',t_0'+t_1'+t_2']\times\mathbb{R}_+.$
Let $0<\delta\leq \min\{\delta_1,\delta_2,\delta_3,\delta_4\}$ with
$$\Xi\left(\frac{1}{2C_1},
2C_1,\frac{\lambda_3}
      {2(C_2 \lambda_3T_1+2)},2C_3\right)\delta_4=\epsilon_0.$$
Then we infer from
Lemmas
\ref{lem_boun}, \ref{lem_lower2} and \ref{lem_the2}
with $T=t_0'+t_1'+t_2'$
that
the local solution $(\phi,\psi,\vartheta)$ satisfies \eqref{3.8} and
\eqref{3.3} for each $t\in[t_0',t_0'+t_1'+t_2']$.
By assuming $0<\delta\leq \min\{\delta_1,\delta_2,\delta_3,\delta_4\}$,
we can repeatedly apply the argument above to extend the local solution
to the time interval $[0,t_0'+T_1]$.
Furthermore, we deduce that
\eqref{3.8} and
\eqref{3.3} hold for each $t\in[t_0',t_0'+T_1]$.
In view of $t_0'+T_1\geq \frac{3}{2}T_1$, we have shown that
the problem \eqref{per} admits a unique solution $(\phi,\psi,\vartheta)$ on
$[0,\frac{3}{2}T_1]$.

\noindent{\bf Step 6}.
We take $0<\delta\leq \min\{\delta_1,\delta_2,\delta_3,\delta_4\}$.
As in Steps 4 and 5, we can find $t_0''\in[t_0'+T_1/2,t_0'+T_1]$ such that
the problem \eqref{per} admits a unique solution $(\phi,\psi,\vartheta)$ on
$[0,t_0''+T_1]$, which satisfies
\eqref{3.8} and
\eqref{3.3} for each $t\in[t_0',t_0''+T_1]$.
Since $t_0''+T_1\geq t_0'+\frac{3}{2}T_1\geq 2T_1$,
we have extended the local solution $(\phi,\psi,\vartheta)$
to $[0,2T_1]$.
Repeating the above procedure, we can then extend the
solution $(\phi,\psi,\vartheta)$ step by step to a global one provided that
$\delta\leq \min\{\delta_1,\delta_2,\delta_3,\delta_4\}$.
Choosing $\epsilon_1=\min\{\delta_1,\delta_2,\delta_3,\delta_4\}$, we then
derive that
the problem \eqref{per} has a unique solution
$(\phi,\psi,\vartheta)$ satisfying
\eqref{3.3}
and
\begin{equation*}
  \inf_{\mathbb{R}_+}\theta(t,\cdot)\geq \min\left\{
\frac{\lambda_2}{C_2 \lambda_2T_1+1},
\frac{\lambda_3}{C_2 \lambda_3T_1+2}
\right\}
\end{equation*}
 for each $t\in[0,\infty)$.

Therefore, we can find constant $C_5$ depending only on
$\inf_{\mathbb{R}_+}(\rho_0,\theta_0)$
  and $\|(\phi_0,\psi_0,\vartheta_0)\|_1$ such that
\begin{equation*}
    \sup_{0\leq t<\infty}\|(\phi,\psi,\vartheta)(t)\|_1^2
    +\int_0^{\infty}
    \left[\|\phi_x(t)\|^2+\|(\psi_{x},\vartheta_{x})(t)\|_1^2\right]dt
    \leq C_5^2,
\end{equation*}
from which the large-time behavior \eqref{thm_3} follows in a standard argument.
This completes the proof of Theorem \ref{thm}.

\bigbreak

\begin{center}
{\bf Acknowledgement}
\end{center}
The authors
express much gratitude to Professor Huijiang Zhao for his support and advice.

\bibliographystyle{siam}

\bibliography{outflow}

\begin{thebibliography}{10}

\bibitem{FLWZMR3260233}
{\sc L.~Fan, H.~Liu, T.~Wang, and H.~Zhao}, {\em Inflow problem for the
  one-dimensional compressible {N}avier-{S}tokes equations under large initial
  perturbation}, J. Differential Equations, 257 (2014), pp.~3521--3553.

\bibitem{HLSMR2730324}
{\sc F.~Huang, J.~Li, and X.~Shi}, {\em Asymptotic behavior of solutions to the
  full compressible {N}avier-{S}tokes equations in the half space}, Commun.
  Math. Sci., 8 (2010), pp.~639--654.

\bibitem{HMS03MR1997442}
{\sc F.~Huang, A.~Matsumura, and X.~Shi}, {\em Viscous shock wave and boundary
  layer solution to an inflow problem for compressible viscous gas}, Comm.
  Math. Phys., 239 (2003), pp.~261--285.

\bibitem{HMS04MR2040072}
\leavevmode\vrule height 2pt depth -1.6pt width 23pt, {\em On the stability of
  contact discontinuity for compressible {N}avier-{S}tokes equations with free
  boundary}, Osaka J. Math., 41 (2004), pp.~193--210.

\bibitem{HQ09MR2514736}
{\sc F.~Huang and X.~Qin}, {\em Stability of boundary layer and rarefaction
  wave to an outflow problem for compressible {N}avier-{S}tokes equations under
  large perturbation}, J. Differential Equations, 246 (2009), pp.~4077--4096.

\bibitem{J99MR1671920}
{\sc S.~Jiang}, {\em Large-time behavior of solutions to the equations of a
  one-dimensional viscous polytropic ideal gas in unbounded domains}, Comm.
  Math. Phys., 200 (1999), pp.~181--193.

\bibitem{J02MR1912419}
\leavevmode\vrule height 2pt depth -1.6pt width 23pt, {\em Remarks on the
  asymptotic behaviour of solutions to the compressible {N}avier-{S}tokes
  equations in the half-line}, Proc. Roy. Soc. Edinburgh Sect. A, 132 (2002),
  pp.~627--638.

\bibitem{KNNZMR2755498}
{\sc S.~Kawashima, T.~Nakamura, S.~Nishibata, and P.~Zhu}, {\em Stationary
  waves to viscous heat-conductive gases in half-space: existence, stability
  and convergence rate}, Math. Models Methods Appl. Sci., 20 (2010),
  pp.~2201--2235.

\bibitem{KNZ03MR2005853}
{\sc S.~Kawashima, S.~Nishibata, and P.~Zhu}, {\em Asymptotic stability of the
  stationary solution to the compressible {N}avier-{S}tokes equations in the
  half space}, Comm. Math. Phys., 240 (2003), pp.~483--500.

\bibitem{KZ08MR2420517}
{\sc S.~Kawashima and P.~Zhu}, {\em Asymptotic stability of nonlinear wave for
  the compressible {N}avier-{S}tokes equations in the half space}, J.
  Differential Equations, 244 (2008), pp.~3151--3179.

\bibitem{KZ09MR2533925}
\leavevmode\vrule height 2pt depth -1.6pt width 23pt, {\em Asymptotic stability
  of rarefaction wave for the {N}avier-{S}tokes equations for a compressible
  fluid in the half space}, Arch. Ration. Mech. Anal., 194 (2009),
  pp.~105--132.

\bibitem{KaMR651877}
{\sc A.~V. Kazhikhov}, {\em On the {C}auchy problem for the equations of a
  viscous gas}, Sibirsk. Mat. Zh., 23 (1982), pp.~60--64, 220.

\bibitem{KS77MR0468593}
{\sc A.~V. Kazhikhov and V.~V. Shelukhin}, {\em Unique global solution with
  respect to time of initial-boundary value problems for one-dimensional
  equations of a viscous gas}, Prikl. Mat. Meh., 41 (1977), pp.~282--291.

\bibitem{LL}
{\sc J.~Li and Z.~Liang}, {\em Some uniform estimates and large-time behavior
  for one-dimensional compressible navier-stokes system in unbounded domains
  with large data}, Preprint, 2014.
  \href{http://arxiv.org/abs/1404.2214}{arXiv:1404.2214}.

\bibitem{Ma01MR1944189}
{\sc A.~Matsumura}, {\em Inflow and outflow problems in the half space for a
  one-dimensional isentropic model system of compressible viscous gas}, Methods
  Appl. Anal., 8 (2001), pp.~645--666.
\newblock IMS Conference on Differential Equations from Mechanics (Hong Kong,
  1999).

\bibitem{MM99MR1682659}
{\sc A.~Matsumura and M.~Mei}, {\em Convergence to travelling fronts of
  solutions of the {$p$}-system with viscosity in the presence of a boundary},
  Arch. Ration. Mech. Anal., 146 (1999), pp.~1--22.

\bibitem{MN00MR1738558}
{\sc A.~Matsumura and K.~Nishihara}, {\em Global asymptotics toward the
  rarefaction wave for solutions of viscous {$p$}-system with boundary effect},
  Quart. Appl. Math., 58 (2000), pp.~69--83.

\bibitem{MN01MR1888084}
\leavevmode\vrule height 2pt depth -1.6pt width 23pt, {\em Large-time behaviors
  of solutions to an inflow problem in the half space for a one-dimensional
  system of compressible viscous gas}, Comm. Math. Phys., 222 (2001),
  pp.~449--474.

\bibitem{NNYMR2356211}
{\sc T.~Nakamura, S.~Nishibata, and T.~Yuge}, {\em Convergence rate of
  solutions toward stationary solutions to the compressible {N}avier-{S}tokes
  equation in a half line}, J. Differential Equations, 241 (2007), pp.~94--111.

\bibitem{NK00MR1793167}
{\sc Y.~Nikkuni and S.~Kawashima}, {\em Stability of stationary solutions to
  the half-space problem for the discrete {B}oltzmann equation with multiple
  collisions}, Kyushu J. Math., 54 (2000), pp.~233--255.

\bibitem{NYZMR2083790}
{\sc K.~Nishihara, T.~Yang, and H.~Zhao}, {\em Nonlinear stability of strong
  rarefaction waves for compressible {N}avier-{S}tokes equations}, SIAM J.
  Math. Anal., 35 (2004), pp.~1561--1597 (electronic).

\bibitem{Q11MR2785974}
{\sc X.~Qin}, {\em Large-time behaviour of solutions to the outflow problem of
  full compressible {N}avier-{S}tokes equations}, Nonlinearity, 24 (2011),
  pp.~1369--1394.

\bibitem{QW09MR2578799}
{\sc X.~Qin and Y.~Wang}, {\em Stability of wave patterns to the inflow problem
  of full compressible {N}avier-{S}tokes equations}, SIAM J. Math. Anal., 41
  (2009), pp.~2057--2087.

\bibitem{QW11MR2765694}
\leavevmode\vrule height 2pt depth -1.6pt width 23pt, {\em Large-time behavior
  of solutions to the inflow problem of full compressible {N}avier-{S}tokes
  equations}, SIAM J. Math. Anal., 43 (2011), pp.~341--366.

\end{thebibliography}

\end{document}